\documentclass[11pt]{article}
\usepackage{latexsym}  
\usepackage{amsmath,amssymb,amsfonts}

\begin{document}

\title{Constellations and their relationship with categories}
\author{Victoria Gould and Tim Stokes}

\date{}
\maketitle

\newcommand{\bea}{\begin{eqnarray*}}
\newcommand{\eea}{\end{eqnarray*}}

\newcommand{\ben}{\begin{enumerate}}
\newcommand{\een}{\end{enumerate}}

\newcommand{\bi}{\begin{itemize}}
\newcommand{\ei}{\end{itemize}}

\newenvironment{proof}{\noindent \textbf{Proof.}\hspace{.7em}}
                   {\hfill $\Box$
                    \vspace{10pt}}

\newcommand{\fqo}{\Subset}
\newcommand{\mc}{\mathcal}
\newcommand{\F}{\mathcal F}
\newcommand{\Z}{\mathbb Z}
\newcommand{\K}{\mathbb K}
\newcommand{\R}{\mathbb R}
\newcommand{\G}{\mathcal G}
\newcommand{\Q}{\mathcal Q}
\newcommand{\C}{\mathcal C}
\newcommand{\I}{\mathcal I}
\newcommand{\ro}{\overline{\mathcal R}}
\newcommand{\Pc}{{\bf P}}
\newcommand{\T}{{\bf T}}
\newcommand{\pr}{^{\prime}}

\newtheorem{thm}{Theorem}[section]
\newtheorem{theorem}[thm]{Theorem}
\newtheorem{lem}[thm]{Lemma}
\newtheorem{pro}[thm]{Proposition}
\newtheorem{proposition}[thm]{Proposition}
\newtheorem{cor}[thm]{Corollary}
\newtheorem{corollary}[thm]{Corollary}
\newtheorem{eg}[thm]{Example}
\newtheorem{conj}{Conjecture}
\newtheorem{dfn}[thm]{Definition}

\newcommand{\fix}{\textsf{Fix}}





\begin{abstract}
Constellations are partial algebras that are one-sided generalisations of categories.  It has previously been shown that the category of inductive constellations is isomorphic to the category of left restriction semigroups.  Here we consider constellations in full generality, giving many examples.  We characterise those small constellations that are isomorphic to constellations of partial functions.  We examine in detail the relationship between constellations and categories, showing the latter to be special cases of the former.  In particular, we characterise those constellations that arise as (sub-)reducts of categories, and show that categories are nothing but two-sided constellations.  We demonstrate that the notion of substructure can be captured within constellations but not within categories.  We show that every constellation $P$ gives rise to a category $\C(P)$, its canonical extension, in a simplest possible way, and that $P$ is a quotient of $\C(P)$ in a natural sense.  We also show that many of the most common concrete categories may be constructed from simpler quotient constellations using this construction.  We characterise the canonical congruences $\delta$ on a given category $K$ (those for which $K\cong \C(K/\delta)$), and show that the category of constellations is equivalent to the category of categories equipped with distinguished canonical congruence.  
\end{abstract}

\maketitle

\section{Introduction}

The ESN Theorem establishes a correspondence between inverse semigroups and certain types of ordered categories called inductive groupoids. In \cite{lawson1}, this setting is broadened in order to establish a correspondence between what in modern terminology are called two-sided restriction semigroups and inductive categories (in fact something more general than this is done).  One-sided restriction semigroups are also of interest, but the lack of two unary operations makes a correspondence with any kind of category impossible, or at least unnatural.  So in \cite{constell}, a one-sided version of a category called a constellation, and, correspondingly, a one-sided version of an inductive category called an inductive constellation, were introduced; the latter were shown to correspond exactly to one-sided (left) restriction semigroups. 

The second section of \cite{constell} briefly concerns itself with general constellations, but the main focus there is on inductive constellations. However, general constellations have interest for their own sake, being one-sided generalisations of categories.  Here we examine general constellations in detail, and especially their relationship with categories.  We begin with some motivation.

Constellations arise very naturally when one considers the possible ways of making sense of function composition.  We keep things simple for the moment by considering only functions amongst the subsets of some fixed set $X$.  For any such function $f:Y\rightarrow Z$ where $Y,Z$ are subsets of $X$, we say 
\bi
\item $Y$ is the {\em domain} of $f$, Dom$(f)$, and it is assumed that $yf$ exists for all $y\in Y$;
\item $Z$ is the {\em codomain} of $f$, Cod$(f)$, and it is assumed that $yf\in Z$ for all $y\in Y$ for which $yf$ exists;
\item Im$(f)=\{xf\mid x\in Y\}$ is the {\em image} of $f$, a subset  of $Z$.
\ei 
Note that we adopt the convention of writing the function on the right of the element it acts on (so $xf$ rather than $f(x)$), since this fits with our adoption of the convention that composition of functions is to be read left-to-right (so that $fg$ is ``first $f$, then $g$").

We distinguish ``functions amongst subsets of $X$" from ``partial functions on $X$".  The latter are functional binary relations (sets of ordered pairs) on $X$, and have no {\em pre-specified} domain or codomain.  However, they do have well-defined domains and images: for such an $f$,
\bi
\item Dom$(f)=\{x\in X\mid (x,y)\in f\, \mbox{ for some }y\in X\}$;
\item Im$(f)=\{y\in X\mid (x,y)\in f\, \mbox{ for some }x\in X\}$.
\ei

There are three possible approaches to defining composites, as well as domain and range operations, in part depending on which of these viewpoints we adopt.
\begin{enumerate}

\item  In the (small) category of subsets of the set $X$, the arrows are the functions $f:Y\rightarrow Z$ where $Y,Z$ are subsets of $X$, with $Y$ the domain of $f$ and $f$ mapping into its codomain $Z$.  Composition of two arrows is defined if and only if the codomain of the first equals the domain of the second.  There are unary domain and range operations which correspond to restricting the identity map to the domain and codomain respectively of a given function.  The result is a category we call $COD_X$, the category of {\em cod-functions} on the non-empty set $X$.

\item  At the other extreme, the partial functions on $X$ may be made into a semigroup under composition (a subsemigroup of the semigroup of binary relations on $X$ under composition), so that {\em all} compositions $fg$ are defined (even if the result is the empty set).  There are again domain and range operations defined in terms of domains and images of the partial functions (codomains now not being defined), although of course the result is not a category (and domain and range operations do not behave symmetrically due to the asymmetric nature of partial functions).  The result is a semigroup equipped with two unary operations that we call ${\mc P}_X$.

\item  There is a third way, intermediate between those just discussed, although it uses the partial function approach of (2) rather than the function with domain and codomain approach of (1).  One may require that the composite $f\cdot g$ of two partial functions be defined if and only if the image of $f$ is a {\em subset} of the domain of $g$.  This is the {\em constellation product}.  There is a natural domain operation, as in the two previous cases.   In this way, we obtain a partial algebra we call $\C_X$, having one partial binary operation and one unary operation; this is one of the main examples of a constellation given in \cite{constell}. 

\end{enumerate}

(We remark in passing that there is in fact a fourth option that we do not consider here, in which one defines the composite of two cod-functions if and only if the codomain of the first is a subset of the domain of the second.  Such a definition gives rise to a type of structure different to the constellations considered here.)

For approach (3), there is also a range operation as in (2) if one wishes.  However, its role is not vital in the way that it is in category theory.  Besides, in a number of settings related to the above, the very notion of a range operation is itself problematic.  For example, one may consider the partial functions on an infinite set that have infinite domains.  Approach (1) can make sense, providing one requires the specification of a (rather arbitrary) infinite codomain containing the image of a given possibly finite-image partial function.  Approach (2) does not work at all since this set of partial functions is not closed under either the semigroup composition or the range operation.   Approach (3) works and seems the most natural: the partial functions in question are closed under the constellation product and the domain operation, but not any obvious range operation.

For a class $\C$ of structures in which there is a notion of structure-preserving mapping, one usually obtains a concrete category in which the objects are members of $\C$ and the arrows are structure-preserving mappings between members of $\C$ (basically, one needs the composite of two structure-preserving mappings to be structure-preserving). Each arrow has associated with it two objects, the source and target of the mapping: the source is the domain and the target contains its image. Then the composite of two arrows is defined if and only if the target of the first equals the source of the second. This is along the lines of approach (1) above.   One can also obtain a constellation structure analogous to approach (3), in which composition of structure-preserving maps is defined whenever the image of the first mapping is contained in the domain of the second.  (On the other hand, approach (2) is difficult to make sense of in general, and may only have a limited analog: for example if one considers an algebra $A$ together with all homomorphisms of its subalebras into it, one obtains a submonoid of the semigroup of partial mappings on $A$.) 

In what follows we shall see that in many concrete settings, the natural category structure can be obtained from the equally natural (yet simpler, being a quotient of the category) constellation structure. Constellations are expressive enough to capture the notion of substructure where it makes sense, even when the corresponding categories cannot.  On the other hand, every category is shown to be a constellation.  So, somewhat paradoxically, constellations are simultaneously more general, more fundamental, and yet more expressive than categories!  
   
In Section 2 of this article, we review the definition and basic properties of constellations as presented in \cite{constell}.  A number of examples of constellations are given, parallelling the familiar examples of concrete categories.  We then present an equivalent definition of constellations that makes clear the fact that constellations generalise categories; indeed categories are shown to be nothing but ``two-sided" constellations.  Those constellations arising from or embeddable in categories are characterised.  A Cayley-style theorem is given for small constellations (constellations whose underlying classes are sets) satisfying a simple property we call normality.  The relevant notions of subalgebra, homomorphism (radiant) and quotient are considered.  

The third section concentrates on the notion of the canonical extension of a constellation $P$ to a category $\C(P)$, which satisfies a suitable universal property.  For constellations $P$ satisfying a natural ``composability" property (satisfied by the main examples), $P$ is shown to be a quotient of $\C(P)$.  It is shown that many of the most familiar concrete categories of mathematics arise as canonical extensions of corresponding (simpler) constellations.  It is shown that the notion of substructure cannot be expressed in the language of categories, but can in the language of constellations.  Those congruences $\delta$ on a category $K$ that give rise to a constellation $P=K/\delta$ for which $K\cong \C(P)$ are described, the so-called canonical congruences (a notion definable for constellations in general), and maximal such congruences are shown to always exist and to give rise to simple (in the relevant sense) quotients.  The category of constellations is shown to be equivalent to the category of categories equipped with distinguished canonical congruence.

We conclude with some open questions.

\section{The basics of constellations}

\subsection{Defining constellations}  \label{axeg}

We begin with some definitions, examples, and basic facts concerning constellations, to some extent reprising material in the early sections of \cite{constell}.
First, we recall the definition of a constellation, generalised to allow classes rather than sets.  

A {\em constellation} is a structure $P$ of signature $(\cdot\ ,D)$ consisting of a class $P$ with a partial binary operation $\cdot$ and unary operation $D$ (denoted $^{+}$ in \cite{constell}) that maps onto the set of {\em projections} $E\subseteq P$, so that $E=\{D(x)\mid x\in P\}$, and such that for all $e\in E$, $e\cdot e$ exists and equals $e$, and for which, for all $x,y,z\in P$:
\bi
\item[] (C1) if $x\cdot(y\cdot z)$ exists then so does $(x\cdot y)\cdot z$, and then the two are equal;
\item[] (C2) $x\cdot(y\cdot z)$ exists if and only if $x\cdot y$ and $y\cdot z$ exist;
\item[] (C3) for each $x\in P$, $D(x)$ is the unique left identity of $x$ in $E$ (i.e. it satisfies $D(x)\cdot x=x$);
\item[] (C4) for $a\in P$ and $g\in E$, if $a\cdot g$ exists then it equals $a$.
\ei
(Note that ``$D(x)$" is one situation in which we write the function on the left, but here $D$ is an operation name, so there should be no confusion.)
Because there is an asymmetry in the definition of constellations, we should properly call constellations as defined here {\em left} constellations, there being an obvious right-handed version that we return to later.  However, we will only rarely be interested in anything other than the left-handed versions, and so generally omit ``left" in what follows.
 
We shall say that the constellation $P$ is {\em small} if $P$ is a set.  Only small constellations were considered in \cite{constell}, although most of the results given in the second section there carry over to the more general case considered here, with only occasional minor changes in terminology.  

Our first observation is that the axioms for constellations may be simplified somewhat.

\begin{pro} \label{const2}
In the definition of a constellation, (C2) may be replaced by 
\bi
\item[] (Const2) if $x\cdot y$ and $y\cdot z$ exist then so does $x\cdot(y\cdot z)$.
\ei
\end{pro}
\begin{proof}
Clearly (C2) implies (Const2).  Conversely, if (Const2) holds, and $x\cdot(y\cdot z)$ exists then so does $(x\cdot y)\cdot z$ by (C1), whence so do both $x\cdot y$ and $y\cdot z$, so (C2) is satisfied.  
\end{proof}

A result that we will make frequent use of is the following, which is Lemma $2.3$ in \cite{constell}.

\begin{lem}  \label{2p3}
For $s,t$ elements of the constellation $P$, $s\cdot t$ exists if and only if $s\cdot D(t)$ exists, and $D(s\cdot t)=D(s)$.
\end{lem}

An important example of a constellation is ${\mc C}_X$ (outlined in (3) above), consisting of partial functions on the set $X$, in which $s\cdot t$ is the composite $s$ followed by $t$ provided Im$(s)\subseteq$ Dom$(t)$, and undefined otherwise, and $D(s)$ is the restriction of the identity map on $X$ to Dom$(s)$.  Another constellation comes from any quasiordered set $(Q,\leq)$: simply define $e=e\cdot f$ whenever $e\leq f$ in $(Q,\leq)$, and let $D(e)=e$ for all $e\in Q$.  These examples were introduced in \cite{constell}.  

A {\em global right identity} in a constellation $P$ is an element $e\in P$ such that $s\cdot e$ exists and equals $s$ for all $s\in P$; it follows that $D(e)=D(e)\cdot e=e$, so $e\in D(P)$.  Every monoid $(M,\ \cdot)$ is a constellation $(M,\ \cdot\ ,D)$ in which $D(a)=1$ for all $a\in M$, as is easily checked, and $1$ is a global right (indeed two-sided!) identity.  

Just as every semigroup may be enlarged to a monoid with the addition of a new identity element, so too every constellation may be enlarged to one in which there is a global right identity: the following may be shown by easy case analyses.

\begin{pro} \label{adjoin}
Let $P$ be a constellation with $1\not\in P$.  Then $P$ can be enlarged to a constellation $P^1=P\cup \{1\}$ with $1$ a global right identity, by setting $s\cdot 1=s$ for all $s\in P$, $1\cdot 1=1$, and letting $1\cdot s$ not be defined for any $s\in P$.
\end{pro}

If $P$ is a constellation, a property an element $s\in P$ may have is that of being {\em composable}: there exists $t\in P$ such that $s\cdot t$ exists.  An element that is not composable is called {\em incomposable}.  If every element of a constellation $P$ is composable, we say $P$ is {\em composable}.  Note that $\C_X$ is composable, as is any monoid or quasiordered set viewed as a constellation.

The class of constellations furnishes the objects in a category whose arrows are {\em radiants}, defined in \cite{constell} to be mappings $\rho:P\rightarrow Q$ (where $P,Q$ are constellations) for which, for all $s,t\in P$, 
\bi
\item if $s\cdot t$ exists in $P$ then $s\rho\cdot t\rho$ exists in $Q$, and then $s\rho\cdot t\rho=(s\cdot t)\rho$, and
\item $D(s)\rho=D(s\rho)$. 
\ei
(In fact a radiant is nothing but a homomorphism of partial algebras as in \cite{gratzer} applied to the case of constellations.)  We call this category {\em the category of constellations}.  

As in \cite{constell}, we say a radiant $\rho:P\rightarrow Q$ is {\em strong} if, whenever $s\rho\cdot t\rho$ exists in $Q$, then so too does $s\cdot t$ exist in $P$.  It is an {\em embedding} if it is strong and injective, and an {\em isomorphism} if it is a surjective embedding.  These definitions are all consistent with standard terminology for the theory of partial algebras as in \cite{gratzer} for example.

\subsection{Examples}  \label{examples}

All of the examples to follow are ``concrete" in the sense that they are set-based, and the elements are certain types of mappings amongst them (generally structure-preserving in some sense).

\begin{eg} The constellation of sets. \label{eg:CSET}
\end{eg}
Let $S$ be the class of sets.  Of course there is a familiar category structure $SET$ associated with $S$, consisting of sets as the objects and maps between them as the arrows; equivalently, taking the ``arrow only'' point of view, the category consists of all possible maps between all possible sets, with the  operations $D$ and $R$ given by specifying $D(f)$ to be the identity map on the domain of $f$ and $R(f)$ the identity map on its codomain, with the partial operation of category composition defined  if and only if domains and codomains coincide.  

We define a constellation $CSET$ from $S$ by analogy with $\C_X$ as above, by taking the elements to be the {\em surjective} functions, with $D$ defined as in the category, but with composition of functions $f\cdot g$ defined if and only if Im$(f)\subseteq$ Dom$(g)$.  The proof that this gives a constellation is very much like the proof that $\C_X$ is a constellation.  It is also composable, like $\C_X$, and its quotient by the largest projection-separating congruence is the same as for $SET$.

Note that the constellation of sets has rather fewer elements than the category of sets, but more products exist amongst these fewer elements.  We return to the details of the relationship between the category and the constellation in this and other cases in the next section.

\begin{eg} The constellation of groups. \label{eg:CGRP} \end{eg}
Let $\G$ be the class of groups.  Of course there is a familiar category structure $GRP$ associated with $\G$: the arrows are all possible homomorphisms between all possible groups, with the partial operation of category composition as well as $D$ and $R$ defined as in $SET$.

We define a composable constellation structure $CGRP$ from $\G$ as follows.  The elements are the surjective homomorphisms between groups, and again composition and $D$ are as in the constellation of sets $CSET$.  

This example generalises widely, for example to any class of algebras of the same type, such as rings, modules, semigroups and so on.

\begin{eg} The constellation of topological spaces.  \label{eg:CTOP} \end{eg}
The category $TOP$ of topological spaces, in which the objects are topological spaces and the arrows are continuous functions between them, likewise has a composable constellation cousin $CTOP$, consisting of the surjective continuous functions between topological spaces, equipped with the obvious domain and composition operations.  The idea generalises to other non-algebraic settings in which the mappings are structure-preserving in some suitable sense (for example, partially ordered sets equipped with surjective order-preserving maps): in each case, there are both category and constellation structures.

\begin{eg} The constellation of partial maps with infinite domain. \label{eg:CXinf}\end{eg}  
Suppose $X$ is an infinite set, and denote by $\C^{\infty}_X$ the set of all partial maps in $\C_X$ that have infinite domains.  Obviously, $\C^{\infty}_X$ is a composable subconstellation of $\C_X$.  Note that no natural type of range operation based on image is available in this example, since images of elements of $\C^{\infty}_X$ can be finite.  A related category would consist of all cod-functions between subsets of $X$ having both infinite domain and codomain.  

We may generalise by replacing ``infinite" by ``large with respect to some bornology on $X$", and we can even admit additional structure, for example algebraic.  Generalising in a different way, we may consider the category $SET^{\infty}$ consisting of all infinite sets together with all mappings between them (noting that images of such mappings need not be infinite), and $CSET^{\infty}$ consisting of all surjective maps with infinite domain equipped with domain and constellation product (under which they are closed).  Similarly, we can consider the category $GROUP^{\infty}$ consisting of all infinite groups with homomorphisms between them and the associated constellation $CGROUP^{\infty}$ consisting of surjective group homomorphisms having infinite domain.

\subsection{Constellations generalise categories}  \label{congen}

We can make more precise the notion that constellations are one-sided generalisations of categories. 

Let $C$ be a class with a partial binary operation.  Recall that $e\in C$ is a {\em right identity} if it is such that, for all $x\in C$, if $x\cdot e$ is defined then it equals $x$; left identities are defined dually.  An {\em identity} is both a left and right identity.  (Note we are not assuming that $e\cdot e$ exists in any of these cases.)

Following \cite{lawson1}, recall that a {\em category} is a class with a partial binary operation satisfying the following:
\bi
\item[] (Cat1) $x\cdot(y\cdot z)$ exists if and only if $(x\cdot y)\cdot z$ exists, and then the two are equal;
\item[] (Cat2) if $x\cdot y$ and $y\cdot z$ exist then so does $x\cdot(y\cdot z)$;
\item[] (Cat3) for each $x\in P$, there are identities $e,f$ such that $e\cdot x$ and $x\cdot f$ exist.
\ei
Note that in \cite{lawson1}, (Cat2) is given in the form 
$$\mbox{if $x\cdot y$ and $y\cdot z$ exist then so does $(x\cdot y)\cdot z$},$$
which is obviously equivalent to the version we give in the presence of (Cat1).

The identities $e,f$ in (Cat3) are easily seen to be unique: for if $e,e\pr$ are identities with $e\cdot x=x$ and $e\pr\cdot x=x$, then $e\cdot (e\pr\cdot x)$ exists, hence so does $(e\cdot e\pr)\cdot x$, whence $e\cdot e\pr$ does, and then $e=e\cdot e\pr=e\pr$ since both are identities.  (Hence $e\cdot e$ exists and equals $e$.)  Similarly for $f$.  In general, we write $D(x)=e$ and $R(x)=f$.  It follows then that the collection of domain elements $D(x)$ (equivalently, range elements $R(x)$) is precisely the collection of identities in the category and does not need {\em a priori} specification: the domain and range operations are defined once uniqueness is shown (or else uniqueness can be redundantly assumed in (Cat3)). 

The following familiar properties hold in a category.

\begin{pro}
Let $x,y$ be elements of a category $C$.
\bi
\item The product $x\cdot y$ exists if and only if $R(x)=D(y)$.
\item If $x\cdot y$ exists then $D(x\cdot y)=D(x)$ and $R(x\cdot y)=R(y)$.
\ei
\end{pro}

Constellations admit an alternative definition that is quite reminiscent of the above way of defining categories.

\begin{pro}  \label{simpler}
Suppose $C$ is a class with a partial binary operation satisfying the following:
\bi
\item[] (Const1) if $x\cdot(y\cdot z)$ exists then $(x\cdot y)\cdot z$ exists, and then the two are equal;
\item[] (Const2) if $x\cdot y$ and $y\cdot z$ exist then so does $x\cdot(y\cdot z)$;
\item[] (Const3) for each $x\in P$, there is a unique right identity $e$ such that $e\cdot x=x$.
\ei
Then $C$ is a constellation in which $D(x)=e$ as in (Const3), and $E$ is the class of right identities of $C$.

Conversely, if $C$ is a constellation  with class of projections $E$, then the class of right identities in $C$ is $E$ and the above three laws hold.
\end{pro}
\begin{proof}
The proof of Proposition \ref{const2} shows that (Const1) and (Const2) are together equivalent to (C1) plus (C2) for binary partial algebras.  

Suppose $C$ satisfies the three new laws.  Now define $D(x)=e$ as in (Const3), and let $E=\{D(x)\mid x\in C\}$.  For $e\in E$, $e=D(e)\cdot e=D(e)$ as $e$ is a right identity, so in particular $e\cdot e=e$, so $D(e)=e\in E$ is idempotent.  We show $E$ is the class of all right identities.  Clearly it is a subclass of them by definition.  But if $e$ is a right identity then again $e=D(e)\cdot e=D(e)$, so $e\in E$.  (C3) and (C4) are now immediate.

Conversely, if $C$ is a constellation with $e$ a right identity, then $e=D(e)\cdot e=D(e)$, so $E$ consists of the class of all right identities in $C$, so (Const3) follows.
\end{proof}

In this new axiomatization for constellations, $D$ is not specified initially, and is determined by the class of all right identities (a class wholly determined by the partial binary operation), parallelling the definition of a category.  

The notion of radiant may be expressed in terms of the new alternative definition also.

\begin{pro}
The map $\rho:P\rightarrow Q$ (where $P,Q$ are constellations) is a radiant if and only if it satisfies, for all $s,t\in P$, 
\bi
\item if $s\cdot t$ exists in $P$ then $s\rho\cdot t\rho$ exists in $Q$, and then $(s\rho)\cdot(t\rho)=(s\cdot t)\rho$, and
\item $e\rho$ is a right identity in $Q$ for every right identity $e$ of $P$.
\ei
\end{pro}
\begin{proof}
Suppose the above conditions are satisfied.  The first is just the first condition in the definition of a radiant.  Now for $e\in D(P)$, $e\rho\in D(Q)$, so $D(e\rho)=e\rho=D(e)\rho$, so the second is also satisfied.  Conversely, if $P$ is a radiant, then the first condition above is satisfied, while if $e$ is a right identity in $P$ then $e=D(e)$ by Proposition \ref{simpler}, so $e\rho=D(e)\rho=D(e\rho)$ which is a right identity in $Q$, so the second condition is satisfied.
\end{proof}

Another advantage of the new formulation is that it makes clearer the fact that constellations generalise categories.

\begin{pro}  \label{rid}
In a category, every right identity is an identity.  Hence every category is a constellation.
\end{pro}
\begin{proof}
Let $e$ be a right identity in the category $C$.  Then $e=D(e)\cdot e=D(e)$, so $e$ is an identity, and the class of identities equals the class of right identities.  So for $x\in C$, $D(x)$ is the unique identity, hence unique right identity, such that $D(x)\cdot x=x$.  
\end{proof} 

As noted earlier, there is an asymmetry in the axioms for constellations not present in the axioms for categories.  We gave the full title of {\em left constellation} to an object satisfying laws (C1)--(C4) above.  We call an object satisfying the following dual axioms a {\em right constellation}: it is a structure $P$ of signature $(\cdot\ ,R)$ consisting of a class $P$ with a partial binary operation $\cdot$ and unary operation $R$ that maps onto the set of {\em projections} $E\subseteq P$, so that $E=\{R(x)\mid x\in P\}$, and such that for all $e\in E$, $e\cdot e$ exists and equals $e$, and for which, for all $x,y,z\in P$:
\bi
\item[] (C1$\pr$) if $(x\cdot y)\cdot z$ exists then so does $x\cdot(y\cdot z)$, and then the two are equal;
\item[] (C2$\pr$) $(x\cdot y)\cdot z$ exists if and only if $x\cdot y$ and $y\cdot z$ exist;
\item[] (C3$\pr$) for each $x\in P$, $R(x)$ is the unique right identity of $x$ in $E$ (i.e. it satisfies $x\cdot R(x)=x$);
\item[] (C4$\pr$) for $a\in P$ and $g\in E$, if $g\cdot a$ exists then it equals $a$.
\ei

By the symmetry of the category axioms, there is an obvious variant of Proposition \ref{rid}, namely that in a category, every left identity is an identity, and so every category is a right constellation also.  There is also a dual version of Proposition \ref{simpler} involving right-handed versions (Const1$\pr$)--(Const3$\pr$) of (Const1)--(Const3)

If $(C,\ \cdot\ ,D,R)$ is a category, then $(C,\ \cdot\ ,D)$ is a left constellation, $(C,\ \cdot\ ,R)$ is a right constellation, and $D(P)=R(P)$.  Indeed there is an easy converse to this, showing that constellations can be viewed as ``one-sided categories". 

\begin{pro}
Let the class $P$ be equipped with a partial binary operation and two unary operations $D,R$ such that $(P,\ \cdot\ ,D)$ is a left constellation, $(P,\ \cdot\ ,R)$ is a right constellation and $D(P)=R(P)$.  Then $(P,\ \cdot\ ,D,R)$ is a category.
\end{pro}
\begin{proof}
We refer to the axioms (Const1)--(Const3) for left constellations and their duals (Const1$\pr$)--(Const3$\pr$) for right constellations. Axioms (Const1) and (Const1$\pr$) together imply (Cat1), and (Cat2) is nothing but (Const2).  Since $D(P)=R(P)$, the left identities of $(P,\cdot)$ are precisely the right identities by Proposition \ref{simpler} and its dual, so (Cat3) obviously holds also, on letting $e=D(x)$ and $f=R(x)$.   
\end{proof}

\subsection{When a constellation arises from a category}

We say a constellation is {\em categorial} if it arises from a category as a reduct (obtained by dropping $R$).  

\begin{pro}  \label{catunique}
Let $P$ be a constellation.  Then $P$ is categorial if and only if, for all $s\in P$ there is a unique $e\in D(P)$ such that $s\cdot e$ exists, and then $R(s)=e$ when $P$ is viewed as a category.
\end{pro}
\begin{proof}
If $P$ is categorial then for each $s\in P$, $s\cdot R(s)$ exists and moreover $R(s)$ is the unique identity (hence unique right identity by Proposition \ref{rid}) for which $s\cdot e$ exists.

Conversely, suppose that for every $s\in P$ there is a unique $e\in D(P)$ such that $s\cdot e$ exists.  Now let $t\in P$ and $f\in D(P)$, with $f\cdot t$ existing.  Then $f\cdot D(t)$ exists by Lemma \ref{2p3}, so since $f\cdot f$ exists, we must have $f=D(t)$ by the uniqueness assumption.  So $f\cdot t=D(t)\cdot t=t$, and $f$ is also a left identity.  So $D(P)$ consists of the identities of $P$, and so for every $s\in P$ there are identities $e,f$ such that $e\cdot s$ and $s\cdot f$ both exist, establishing (Cat3).  (Cat2) is immediate.  It remains to prove (Cat1).

If $s,t,u\in P$ are such that $s\cdot(t\cdot u)$ exists, then $(s\cdot t)\cdot u$ exists also, by (Const1), and they are equal.  Conversely, suppose $s,t,u\in P$ are such that $(s\cdot t)\cdot u$ exists.  Then $(s\cdot t)\cdot D(u)$ exists by Lemma \ref{2p3}.  But there exists (unique) $e\in D(P)$ such that $t\cdot e$ exists, so $s\cdot t=s\cdot(t\cdot e)$ exists, so $(s\cdot t)\cdot e$ exists by (Const1).  So again by the uniqueness assumption, $e=D(u)$, and so $t\cdot D(u)$ exists, so $t\cdot u$ exists by Lemma \ref{2p3}.  But $s\cdot t$ exists, so by (Const2), $s\cdot(t\cdot u)$ exists and $s\cdot(t\cdot u)=(s\cdot t)\cdot u$ also.  So (Cat1) holds.
\end{proof}

If $f:K\rightarrow L$ is a functor, then it is a radiant when the categories $K$ and $L$ are viewed as constellations (so in particular, if $\rho$ above is a category isomorphism, then it is a constellation isomorphism as well).  Indeed the converse is also true.

\begin{pro} \label{radfunct}
If $K,L$ are categories then $\rho:K\rightarrow L$ is a radiant if and only if it is a functor.
\end{pro}
\begin{proof}
If $P,Q$ are categorial constellations and $\rho:P\rightarrow Q$ is a radiant, then for each $s\in P$, there is a unique $e\in D(P)$ such that $s\cdot e$ exists (namely $R(s)$ when $P$ is viewed as a category), so $s\rho=(s\cdot e)\rho=(s\rho)\cdot (e\rho)$ where $e\rho\in D(Q)$, so by uniqueness, $e\rho$ is the unique $f\in D(Q)$ such that $(s\rho)\cdot f$ exists, that is, $R(s\rho)=e\rho=R(s)\rho$.  The converse was dealt with above.
\end{proof}

It follows that the the class of categorial constellations is a full subcategory of the category of all constellations.

We next obtain a description of those constellations embeddable in categories.  First note that any subconstellation of a categorial constellation must satisfy (Cat1) rather than just (C1).  For if $P$ is a subconstellation of categorial $C$, and for some $s,t,u\in P$, $(s\cdot t)\cdot u$ exists in $P$, then it exists in $C$ also and hence must equal $s\cdot (t\cdot u)$; but $t\cdot u\in P$ since it is a subconstellation, whence $s\cdot(t\cdot u)\in P$ (where it must equal $(s\cdot t)\cdot u$).  

We shall need two preliminary lemmas, the first of which is a variant of Proposition \ref{catunique}.

\begin{lem} \label{adjoin*}
Let $P$ be a constellation satisfying (Cat1). Let $P^*=P\cup \{\ast \}$ where $\ast\notin P$ and extend the partial binary operation in $P$ by putting $x\cdot \ast=x$ for all non-composable $x$ and $\ast\cdot \ast=\ast$.
Then $P^*$ is a constellation with $D(P^*)=D(P)\cup \{ \ast\}$.
\end{lem}

\begin{proof} 
Let $x,y\in P^*$. If $\ast\cdot(x\cdot y)$ exists, then clearly $x\cdot y=\ast$, so that $x=y=\ast$ and $(\ast\cdot x)\cdot y=\ast$ exists and equals $\ast\cdot(x\cdot y)$.

If  $x\cdot (\ast \cdot y)$ exists, then $y=\ast$ and $\ast\cdot y=\ast$. Hence $(x\cdot \ast)\cdot y=x$ and equals $x\cdot (\ast\cdot y)$.

Suppose now that $x,y\in P$. If $x\cdot (y\cdot \ast)$ exists then $y$ is non-composable and $y\cdot \ast=y$, so that $x\cdot y$ exists. It follows from (Cat1) that $x\cdot y$ is non-composable so that $(x\cdot y)\cdot \ast=x\cdot y$ exists and equals $x\cdot (y\cdot \ast)$. We have thus shown that (Const1) holds.

A similar (but more straightforward) case by case analysis verifies (Const2). Condition (Const3) is clear.
\end{proof}

The following is now clear from Proposition \ref{catunique} and Lemma \ref{adjoin*}.

\begin{cor}  \label{P1cat}
Let $P$ be a constellation satisfying (Cat1) in which for each $s\in P$ there is at most one $e\in D(P)$ for which $s\cdot e$ exists.  Then $P^*$ as in Proposition \ref{adjoin*} is a categorial constellation.
\end{cor}
%
%
%
%
%

\begin{pro}  \label{catcon}
The constellation $P$ is embeddable in a category if and only if it satisfies (Cat1) and for each $s\in P$, there is at most one $e\in D(P)$ for which $s\cdot e$ exists.
\end{pro}
\begin{proof}
If $P$ embeds in a category, then (Cat1) must hold, and for each $s\in P$ there can be at most one $e\in D(P)$ for which $s\cdot e$ exists.  The converse is immediate from Corollary \ref{P1cat}.
\end{proof}

Proposition \ref{catunique} shows that for any constellation $P$, (Cat1) follows on the assumption that for all $s\in P$ there is a unique $e\in D(P)$ such that $s\cdot e$ exists.  However, (Cat1) cannot be omitted in Proposition \ref{catcon}.  Consider any set $X$ having non-empty proper subset $Y$; then $\C_X^Y$ consisting of all partial functions in $\C_X$ having domain $Y$ is easily seen to be a constellation under the same operations as those on $\C_X$.  For each $s\in \C_X^Y$, there is obviously at most one $e\in D(\C_X^Y)$ for which $s\cdot e$ exists (since $|D(\C_X^Y)|=1$).  However, it does not generally satisfy (Cat1), and hence is not generally embeddable in a category.  For example, if $X=\{1,2\}$ and $Y=\{1\}$, let $s=\{(1,1)\}$ and $t=\{(1,1),(2,2)\}$, the only element of $D(\C_X^Y)$; it is readily verified that $(s\cdot t)\cdot s=s$ but that $t\cdot s$ is not defined, so nor is $s\cdot(t\cdot s)$.  (This example also arises from the poset $P=\{e,f\}$ in which $e\leq f$: $(e\cdot f)\cdot e$ exists but $e\cdot (f\cdot e)$ does not.)

Every small category $C$ may be viewed as a bi-unary semigroup $S$, consisting of ground set $C\cup \{0\}$ where $0\not\in C$, in which one defines $ab=a\cdot b$ if $a\cdot b$ exists in $C$, and $ab=0$ otherwise.  However, nothing similar can be done with constellations, since, as we have just seen, there can be elements $x,y,z$ of a constellation $C$ for which $(x\cdot y)\cdot z$ is defined yet $x\cdot(y\cdot z)$ is not.  So a similar approach of defining a total binary operation on a constellation augmented by a new element $0$ may not yield an associative operation, and so there is no way to view many constellations as equivalent to any form of semigroup (even one defined on a class rather than a set).

\subsection{Subconstellations}

If $P$ is a partial algebra then non-empty $Q\subseteq P$ is a said to be a {\em subalgebra} if it is closed under the partial operations on $P$.  For a constellation $P$, this means that non-empty $Q\subseteq P$ is a subalgebra if for $a,b\in Q$, if $a\cdot b$ exists in $P$ then it lies in $Q$, and $D(a)\in Q$ for all $a\in Q$.  The following is easily shown.

\begin{pro}  \label{subisconst}
If $P$ is a constellation with subalgebra $Q$, then $Q$ is a constellation under the restrictions of the operations of $P$.
\end{pro}
%
%
%
%

For this reason, we refer to a subalgebra $Q$ of a constellation $P$ as a {\em subconstellation}.  By (C4), $D(P)$ is a subconstellation of the constellation $P$.

\begin{pro}  \label{compembed}
The constellation $P$ is a subconstellation of $P^1$ as in Proposition \ref{adjoin}.  Hence every constellation arises as a subconstellation of a constellation with global right identity, and in particular of a composable constellation.
\end{pro}

\begin{pro}  \label{subembed}
The image of a strong radiant $\rho:P\rightarrow Q$ is a subconstellation of $Q$, isomorphic to $P$ if $\rho$ is an embedding.
\end{pro}
\begin{proof}
It follows from the general theory of partial algebras as in \cite{gratzer} that the image $P\rho$ of $P$ in $Q$ is a subalgebra, hence a subconstellation by Proposition \ref{subisconst}.
\end{proof}

On the other hand, if $P$ is a subconstellation of $Q$ then the inclusion map $i:P\rightarrow Q$ is easily seen to be an embedding.

\subsection{Normality and a Cayley theorem}

There is a natural quasiordering on the projections in a constellation $P$.

\begin{pro}  \label{parnormal}
If $P$ is a constellation, define the relation $\lesssim$ on $D(P)$ by $e\lesssim f$ if and only if $e\cdot f$ exists.  Then $\lesssim$ is a quasiorder, and is a partial order if and only if
$\mbox{for all }e,f\in D(P),\mbox{ if $e\cdot f$ and $f\cdot e$ exist, then }e=f.$
\end{pro}
\begin{proof}
For $e,f,g\in D(P)$, if $e\lesssim f$ and $f\lesssim g$ then $e=e\cdot f$ and $f=f\cdot g$ exist, so $e\cdot (f\cdot g)$ exists by (C2), hence so does $(e\cdot f)\cdot g=e\cdot g$ by (C1).  Reflexivity is obvious, so $\lesssim$ is a quasiorder.  It is a partial order if and only if it is antisymmetric, which translates into the condition that for all $e,f\in D(P)$, if $e\cdot f$ and $f\cdot e$ exist, then $e=f$.
\end{proof}

We call the quasiorder $\lesssim$ above the {\em standard quasiorder} on $D(P)$; let $\approx$ be the associated equivalence relation, so that $e\approx f$ if and only if $e\lesssim f$ and $f\lesssim e$.  The condition ensuring $\lesssim$ is a partial order (that is, $\approx$ is the identity map) is called {\em normality}, and we say $P$ is a {\em normal constellation} in this case.  A normal constellation has at most one global right identity element.  A quasiordered set is normal when viewed as a constellation if and only if it is a poset.  A subconstellation of a normal constellation is normal.  Every categorial constellation is normal since $e\lesssim f$ if and only if $e=f$ (where $e,f$ are identities). 

In each example of a constellation given in Section \ref{examples}, the standard quasiorder on $D(P)$ corresponds to the notion of being a substructure in the appropriate sense, so all the examples given there are normal.  From Theorem \ref{equiv}, every inductive constellation is normal (since each embeds as a constellation in some $\C_X$), but the converse is certainly false.  For example, although $\C_X$ is inductive, $\C^{\infty}_X$ is not, as $D(\C^{\infty}_X)$ is not a semilattice and so condition (I) for inductive constellations as in \cite{constell} fails.

If $(Q,\leq)$ is a quasiordered set, then as noted earlier, we may view it as a constellation in which $D(Q)=Q$ by defining $e=e\cdot f$ if $e\leq f$ in $(Q,\leq)$, and letting $D(e)=e$ for all $e\in Q$.  Then $\leq$ is nothing but the standard quasiorder on $D(Q)=Q$ viewed as a constellation.  Indeed there is an isomorphism between the categories of quasiordered sets and constellations $Q$ in which $D(Q)=Q$, since quasiorder-preserving maps are nothing but radiants under this correspondence, as is easily seen.  This isomorphism specialises to one between normal constellations in which $D(Q)=Q$ and partially ordered sets.

The uniqueness assumption in (Const3) in our new definition of constellations may be omitted in the normal case, generalising the case of categorial constellations.

\begin{pro}
Suppose $C$ is a class with a partial binary operation satisfying (Const1) and (Const2), and for all right identities $e$ and $f$, the normality property is satisfied.  Then (Const3) is equivalent to 
\bi
\item[] (Const3A) for each $x\in P$, there is a right identity $e$ such that $e\cdot x=x$.
\ei
\end{pro}
\begin{proof}
Suppose (Const3A) holds.  Suppose $e\cdot x=f\cdot x=x$ where $e,f$ are right identities.  Then $e\cdot(f\cdot x)$ is defined, hence so is $(e\cdot f)\cdot x$ by (Const1), so $e\cdot f$ exists.  By symmetry, so does $f\cdot e$.  Hence $e=f$, and (Const3) holds.
\end{proof}

So (Const3A) may be used in place of (Const3) to define normal constellations. 

Recall the constellation of partial functions $\C_X$ on the set $X$, defined above.  
Let us say a small constellation $P$ is {\em functional} if it embeds in ${\mc C}_X$ for some set $X$.  In particular, inductive constellations are functional (Proposition $3.9$ of \cite{constell}).  By Proposition \ref{subembed} and the comment following it, $P$ is functional if and only if it is (isomorphic to) a subconstellation of ${\mc C}_X$ for some set $X$.

\begin{thm}  \label{equiv}
A small constellation is functional if and only if it is normal.
\end{thm} 
\begin{proof}
If the small constellation $P$ is functional then for $e,f\in D(P)$, $e\lesssim f$ can be interpreted as set inclusion of the domain of $e$ in the domain of $f$; of course, set inclusion is a partial order, proving normality.

Conversely, if the small constellation $P$ is normal then the strong radiant $\rho:P\rightarrow {\mc C}_P$ given by $s\mapsto \rho_s$, where $x\rho_s=x\cdot s$ if it exists and undefined otherwise (as in Proposition $2.6$ of \cite{constell}) is also injective, hence an embedding.  For suppose $\rho_s=\rho_t$. Then $D(s)\cdot t$ exists and equals $s$ (since $D(s)\cdot s$ exists and equals $s$), so $D(s)\cdot D(t)$ exists by Lemma \ref{2p3}.  Similarly, $D(t)\cdot D(s)$ exists, and so $D(s)=D(t)$ by normality.  Hence $s=D(s)\cdot t=D(t)\cdot t=t$.  So $P$ is functional.
\end{proof}

There is a similar well-known Cayley-type theorem for categories: every small category is embeddable in the small category of subsets of some set equipped with the cod-functions between them.

\subsection{Congruences, quotients and homomorphism theorems}

There are natural notions of congruence and quotient for constellations, which are again special cases of more general notions for partial algebras as in \cite{gratzer}.  

Given a constellation $C$ we say the equivalence relation $\delta$ is a {\em congruence} if, whenever $(s_1,s_2)\in \delta$ and $(t_1,t_2)\in \delta$, and both $s_1\cdot t_1$ and $s_2\cdot t_2$ are defined, it is the case that that $(s_1\cdot t_1, s_2\cdot t_2)\in \delta$, and $(D(s_1),D(s_2))\in \delta$; it is a {\em strong congruence} if it has the property that $s_1\cdot t_1$ is defined if and only if $s_2\cdot t_2$ is, for all such $s_1,t_1,s_2,t_2$.  The kernel $\delta$ of the radiant $\rho:P\rightarrow Q$ (defined to be an equivalence relation on $P$ in the usual way) is a congruence but may not be strong.

Given a congruence $\delta$ on the constellation $P$, let $[x]$ denote the $\delta$-class containing $x\in C$.  Consistent with its definition for general partial algebras, we may define the {\em quotient} $P/\delta$ by setting $[s]\cdot[t]=[u]$ if $s_1\cdot t_1=u_1$ for some $s_1\in [s]$, $t_1\in [t]$ and $u_1\in [u]$; this is easily seen to be well-defined.  (If $\delta$ is a strong congruence, we may simply define $[s]\cdot [t]=[s\cdot t]$ whenever $s\cdot t$ exists.)  We also define $D([s])=[D(s)]$ for all $s\in S$.  

Given a congruence $\delta$ on the constellation $C$, the natural map taking $x$ to $[x]$ for all $x\in C$ is a partial algebra homomorphism $C\rightarrow C/\delta$, strong if and only if $\delta$ is.  In important cases, the quotient $C/\delta$ is in fact a constellation also.  Indeed it is possible to write down conditions on a congruence $\delta$ defined on the constellation $P$ that characterise when $P/\delta$ is a constellation, based on any of the equivalent sets of conditions that can be used to define constellations given above.  These conditions need not be met in all cases, as the following example shows.  Let $P=\{s,t_1,t_2,u\}$ be the constellation in which all four elements are idempotent, and $s\cdot t_1=s$, $t_2\cdot u=t_2$, with no other products defined.  (This arises from the partial order on the set $P$ given by $s\leq t_1$ and $t_2\leq u$.)   Then the equivalence relation $\delta$ with classes $a=\{s\}, b=\{t_1,t_2\}, c=\{u\}$ is easily seen to be a congruence, and $a\cdot b=a$, $b\cdot c=b$.  Note that $a\cdot(b\cdot c)=a$ exists yet $(a\cdot b)\cdot c$ does not, so (C1) fails.  

We say the congruence $\delta$ on the constellation $P$ is {\em right strong} if, whenever $[s]\cdot [t]$ exists in $P/\delta$, there exists  $s_1\in [s]$ such that $s_1\cdot t$ exists in $P$.  Put another way, $\delta$ is right strong if, whenever $s\cdot t$ exists in $P$, then for all $t_1\ \delta\ t$ there exists $s_1\ \delta\ s$ for which $s_1\cdot t_1$ exists.  Obviously, if $\delta$ is strong then it is right strong.

\begin{pro}  \label{rsquot} 
Let $P$ be a constellation, $\delta$ a right strong congruence on $P$.  Then $Q=P/\delta$ is a constellation for which $D(Q)=\{[e]\mid e\in D(P)\}$, and $Q$ is composable if $P$ is.
\end{pro}
\begin{proof}
Suppose $\delta$ is a right strong congruence on $P$, with $[x]$ the $\delta$-class in $Q$ containing $x\in P$, as usual.

Suppose $[a]\cdot([b]\cdot[c])$ exists.  Then there exist $b_1,d$ such that $b_1\in [b]$ and $b_1\cdot c$ exists, and also $a_1\in [a]$ for which $a_1\cdot (b_1\cdot c)$ exists.  So $(a_1\cdot b_1)\cdot c$ exists and equals $a_1\cdot (b_1\cdot c)$. So $([a]\cdot[b])\cdot[c]$ exists, and since it overlaps with $[a]\cdot([b]\cdot[c])$, they are equal.  So (C1) holds.

Suppose $[a]\cdot [b]$ and $[b]\cdot [c]$ exist.  Then there is $b_1\in P$ for which $[b]=[b_1]$ and $b_1\cdot c$ exists, and then there is $a_1\in P$ for which $[a_1]=[a]$ and $a_1\cdot b_1$ exists.  Hence $a_1\cdot (b_1\cdot c)$ exists, so $[a]\cdot ([b]\cdot[c])$ exists as well.  So (Const2) holds.  

For $e\in D(P)$, $[e]\cdot[e]=[e\cdot e]=[e]$, and if $[s]\cdot[e]$ exists then $s_1\cdot e$ exists for some $s_1$ such that $[s]=[s_1]$, so $[s]\cdot[e]=[s]$. So $[e]$ is a right identity in $Q$.  Conversely, if $[s]$ is a right identity in $Q$, then since $[D(s)]\cdot [s]=[D(s)\cdot s]=[s]$ exists, it must equal $[D(s)]$, so $[s]=[D(s)]$.  So the right identities in $Q$ are precisely the $[e]$ where $e$ is a right identity in $P$.

Now $D([a])\cdot[a]=[D(a)]\cdot[a]=[D(a)\cdot a]=[a]$ since $D(a)\cdot a$ exists.  If also $[e]\cdot [a]$ exists and equals $[a]$, then $x\cdot a$ exists and $x\cdot a\ \delta\ a$ for some $x\in [e]$, so $D(a)=D(x\cdot a)=D(x)\ \delta\ D(e)=e$, and so $[a]=[e]$.  Hence (Const3) holds.  So $Q=P/\delta$ is a constellation.  

If $P$ is composable then for all $[s]\in Q$ (where $s\in P$), there is $t\in P$ such that $s\cdot t$ exists, so $[s]\cdot [t]$ exists also, and so $Q$ is composable.
\end{proof}

In particular, every quotient based on a right strong congruence on a constellation is a constellation, although the converse is false.  Consider the partially ordered set $Q=\{e,f,g\}$ in which $e\leq f$ is the only non-reflexive relation between elements.  Viewed as a constellation, $Q$ has a non-right strong congruence $\delta$ such that the associated partition is $\{e\}, \{f,g\}$, and clearly $Q/\delta$ is a constellation even though $\delta$ is not right strong: $g\in [f]$ and $e\cdot f$ exists, yet $e\cdot g$ does not exist.  

Another important property that a congruence $\delta$ on a constellation $P$ can have is that it {\em separates projections}: for all $e,f\in D(P)$, if $e\ \delta\ f$ then $e=f$.  A congruence on a category, as it is traditionally defined, is nothing but a projection-separating congruence on its constellation reduct that also respects the range operation.  Note that if $P$ is normal and $\delta$ is a strong congruence on $P$ then for all $e,f\in P$, if $[e]=[f]$ where $e,f\in D(P)$, then $[e]\cdot[f]$ and $[f]\cdot [e]$ exist, so $e\cdot f$ and $f\cdot e$ exist in $P$, so by normality of $P$, $e=f$.  Of course $\delta$ is trivially right strong.  In general we have the following. 

\begin{pro}  \label{seproj}
Let $P$ be a constellation, $\delta$ a congruence on $P$ that separates projections.  
Then $\delta$ is right strong.
\end{pro}
\begin{proof}
If $[s]\cdot [t]$ exists then there are $s_1\ \delta\ s$ and $t_1\ \delta\ t$ such that $s_1\cdot t_1$ exists, so $s_1\cdot D(t_1)$ exists.  But $D(t_1)\ \delta\ D(t)$, so $D(t_1)=D(t)$, and so $s_1\cdot D(t)$ exists, so $s_1\cdot t$ exists.  (This uses Lemma \ref{2p3} twice.) 
\end{proof}

We have the following useful consequence of Propositions \ref{rsquot} and \ref{seproj}.

\begin{cor}  \label{sepquot}
Let $P$ be a constellation, $\delta$ a congruence on $P$ that separates projections.  
Then $Q=P/\delta$ is a constellation for which $D(Q)=\{[e]\mid e\in D(P)\}$, and $Q$ is composable if $P$ is.
\end{cor}

The diagonal relation $\triangle$ on a constellation $P$ is a strong and projection-separating congruence, the {\em trivial congruence} on $P$.  Obviously $P/\triangle\cong P$.

Consider the full congruence $\triangledown$ on the constellation $P=\{e,f\}$ in which only the products $e\cdot e$ and $f\cdot f$ exist.  This is right strong but neither strong nor projection-separating. 


As noted earlier, the definition of a radiant between constellations is simply the definition of homomorphism of partial algebras as in \cite{gratzer} applied to constellations, and use of the term ``strong" here is consistent with usage  in the general setting of \cite{gratzer} also.  A further special type of homomorphism covered in \cite{gratzer} applies to constellations as follows: a radiant $\rho:P\rightarrow Q$ is {\em full} if for all $s,t\in P$ for which $(s\rho)\cdot (t\rho)$ exists and is in the image of $\rho$, there are $s\pr,t\pr\in P$ such that $s\rho=s\pr\rho$, $t\rho=t\pr\rho$ and $s\pr\cdot t\pr$ exists in $P$.  The next result again follows from the general theory of partial algebras and appears in \cite{gratzer}.

\begin{pro}  \label{natmap}
Suppose $P$ is a constellation having congruence $\delta$ such that $P/\delta$ is a constellation.  Then the natural map $P\rightarrow P/\delta$ is a full radiant.
\end{pro}

There is a converse to this; see page 98 of \cite{gratzer}, where it is noted that the Homomorphism Theorem carries over to partial algebras in the following way.  

\begin{pro}  \label{semiquot}
Let $P_1,P_2$ be constellations.  If the surjective radiant $\rho:P_1\rightarrow P_2$ is full then $P_1/\ker(\rho)\cong P_2$. 
\end{pro}

A radiant $\rho:P\rightarrow Q$ {\em separates projections} if for all $e,f\in D(P)$, if $e\rho=f\rho$ then $e=f$.  The following three results are clear.

\begin{pro}  \label{sepcorresp}
If $\rho:P_1\rightarrow P_2$ is a full surjective radiant that separates projections then $\ker(\rho)$ separates projections.  If $\delta$ is a congruence on the constellation $P$ that separates projections (whence $P/\delta$ is a constellation by Corollary \ref{sepquot}) then the natural map $\nu:P\rightarrow P/\delta$ is a full surjective radiant that separates projections.
\end{pro}

\begin{pro}  \label{maxprojsep}
There is a largest projection-separating congruence $\delta$ on every constellation $P$, given by $s\ \delta\ t$ if and only if $D(s)=D(t)$, so that $D(P/\delta)=P/\delta$, and $[e]\lesssim [f]$ in $D(P/\delta)$ if and only if there exists $s\in P$ for which $D(s)=e$ and $s\cdot f$ exists (so if $P$ is a category, there is $s\in P$ such that $D(s)=e$ and $R(s)=f$).
\end{pro}

For example, the largest projection-separating congruence on $SET$ gives a quotient that is the full quasiorder on all sets since for any two sets $X,Y$, there is a function with domain $X$ that maps into $Y$.

\begin{cor}  \label{maxsimpproj}
The constellation $P$ has no non-trivial projection-separating congruences if and only if $D(P)=P$.
\end{cor}

There is a version of the correspondence theorem for constellations with respect to projection-separating congruences.  For $\delta$ an equivalence relation on the constellation $P$, denote by $[s]_{\delta}$ the $\delta$-class containing $s\in P$ (a notation needed below since more than one relation is being considered).

\begin{lem}  \label{seprad}
For $P$ a constellation and $\delta,\gamma$ two projection-separating congruences on $P$ for which $\delta\subseteq \gamma$, the mapping $\rho:P/\delta\rightarrow P/\gamma$ given by $[s]_{\delta}\rho=[s]_{\gamma}$ is a full projection-separating surjective radiant.  \end{lem}
\begin{proof}
%
That $\rho$ is a full surjective partial algebra homomorphism follows from the general theory of partial algebras.  If $\gamma$ and hence $\delta$ are congruences that are projection-separating, then they are right strong and so $P/\delta$ and $P/\gamma$ are constellations in which the domain elements are of the form $[e]_{\delta}$ (or $[e]_{\gamma}$) where $e\in D(P)$ by Proposition \ref{rsquot}.  If $[e]_{\delta}\rho=[f]_{\delta}\rho$ for $e,f\in D(P)$, then $[e]_{\gamma}=[f]_{\gamma}$ so $e=f$ since $\gamma$ is projection-separating, and so $[e]_{\delta}=[f]_{\delta}$.
\end{proof}

\begin{pro}  \label{trans}
Let $P$ be a constellation with $\delta$ a projection-separating congruence on it, and let $Q=P/\delta$.  Then for any projection-separating congruence $\epsilon$ on $Q$, there exists a projection-separating congruence $\gamma$ on $P$ that contains $\delta$ and such that $P/\gamma\cong Q/\epsilon$, given by $s\ \gamma\ t$ if and only if $[s]_{\delta}\ \epsilon\ [t]_{\delta}$.  Conversely, if $\gamma$ is a projection-separating congruence on $P$ containing $\delta$, then there is a projection-separating congruence $\epsilon$ on $Q$ for which $P/\gamma\cong Q/\epsilon$, given by $[s]_{\delta}\ \epsilon\ [t]_{\delta}$ if and only if $s\ \gamma\ t$.
\end{pro}
\begin{proof}  Let $P,Q,\delta$ be as stated.

Let $\epsilon$ be a projection-separating congruence on $Q$.  Let $\nu_1:P\rightarrow P/\delta$ and $\nu_2:Q\rightarrow Q/\epsilon$ be the natural maps (which are full surjective projection-separating radiants by Proposition \ref{natmap}), and let $\nu:P\rightarrow Q/\epsilon$ be their composite, also a full surjective radiant that separates projections.  Let $\gamma=\ker(\nu)$.  Then $(s,t)\in\gamma$ if and only if $[s]_{\delta}\ \epsilon\ [t]_{\delta}$, so $\delta\subseteq \gamma$, $\gamma$ is a congruence that separates projections by Proposition \ref{sepcorresp}, and $P/\gamma\cong Q/\epsilon$ by Proposition \ref{semiquot}.  

Conversely, suppose $\gamma$ is a projection-separating congruence on $P$, with $\delta\subseteq \gamma$.  We have the natural maps $\nu:P\rightarrow P/\gamma$ and $\nu_1:P\rightarrow P/\delta$.  Since $\delta\subseteq\gamma$, we may define $\rho:P/\delta\rightarrow P/\gamma$  by setting $[s]_{\delta}\: \rho=[s]_{\gamma}$, a full surjective radiant that separates projections by Lemma \ref{seprad}, so $\epsilon=\ker(\rho)$ is a projection-separating congruence on $Q$ such that $Q/\epsilon\cong P/\gamma$ by Proposition \ref{semiquot}. 
\end{proof}

\section{The canonical extension of a constellation to a category}

\subsection{Definition and basic properties}

Let $P$ be a constellation.  Let $\C(P)$ be the structure consisting of the ordered pairs $(s,e)$, where $s\in P$, $e\in D(P)$, and $s\cdot e=s$.  On $\C(P)$, define 
$$(s,e)\circ (t,f)=(s\cdot t,f)$$
providing $e= D(t)$ (in which case $s\cdot t$ exists since both $s\cdot e$ and $e\cdot t$ exist, whence so does $s\cdot (e\cdot t)=s\cdot t$).  Also define $D((s,e))=(D(s),D(s))$ and $R((s,e))=(e,e)$. 

\begin{pro}  \label{CPran}
For any constellation $P$, $(\C(P),\circ,D,R)$ is a category.
\end{pro}

\begin{proof}
Suppose $(s,e)$ is an identity in $\C(P)$, and suppose $x\cdot s$ exists for some $x\in P$.   Hence $x\cdot D(s)$ exists, and so $(x,D(s))\in \C(P)$.  Now $(x,D(s))\circ (s,e)$ exists, and equals $(x\cdot s,e)$, which must equal $(x,D(s))$ since $(s,e)$ is an identity, so $x\cdot s=x$.  So $s$ is a right identity in $P$, and $e=D(s)=s$.  So $(s,e)=(e,e)$ for some right identity $e$ of $P$.  Conversely, it is easy to check that $(e,e)\in \C(P)$ (where $e$ is a right identity of $P$) is a two-sided identity of $\C(P)$, so such elements are precisely the identities of $\C(P)$.  

For $(x,e),(y,f),(z,g)\in \C(P)$, the following are equivalent: $(x,e)\circ((y,f)\circ (z,g))$ exists; $e=D(y)$ and $f=D(z)$; $((x,e)\circ (y,f))\circ (z,g)$ exists (since $D(y\cdot z)=D(y)$).  Both products are easily seen to equal $(x\cdot (y\cdot z),g)$, so (Cat1) holds.   The second condition is equivalent to $(x,e)\circ (y,f)$ and $(y,f)\circ (z,g)$ existing, proving (Cat2).  Finally, for $(x,e)\in\C(P)$, note that $(D(x),D(x))\circ (x,e)=(x,e)$ and $(x,e)\circ (e,e)=(x,e)$, establishing (Cat3).
\end{proof}

We call $(\C(P),\circ,D,R)$ as in Proposition \ref{CPran} the {\em canonical extension} of the constellation $P$.  It may be viewed as a generalisation of a construction given in \cite{agree} that builds a category from an RC-semigroup satisfying the right congruence condition (a certain type of unary semigroup generalising right restriction semigroups).  Note that if $P$ is small, then $\C(P)$ is also.

If $K$ is a category and $P$ its constellation reduct, then it is easy to see that $K\cong \C(P)$.  This is a special case of the following fact.

\begin{pro}  \label{radim}
For every constellation $P$, the mapping $\rho:\C(P)\rightarrow P$ given by $(s,e)\rho=s$ is a radiant.   If $P$ is composable then $\rho$ is full and surjective, so $P$ is a constellation quotient of $\C(P)$; in particular, if $P$ is categorial, then $P\cong \C(P)$.
\end{pro}

\begin{proof}
Suppose $(s,e)\circ (t,f)$ exists.  Then trivially $(s,e)\rho\cdot (t,f)\rho=s\cdot t=(s\cdot t,f)\rho=((s,e)\circ (t,f))\rho$ exists.  Moreover $D((s,e))\rho=(D(s),D(s))\rho=D(s)=D((s,e)\rho)$.  So $\rho$ is a radiant.

Now if $P$ is composable, then for $x\in P$ there exists $y\in P$ such that $x\cdot y$ exists.  Then $x\cdot D(y)$ exists by Lemma \ref{2p3}, and so $(x,D(y))\in \C(P)$, with $(x,D(y))\rho=x$, so $\rho$ is surjective.  Suppose now that $(x,e),(y,f)\in \C(P)$ with $(x,e)\rho\cdot (y,f)\rho$ existing.  Then $x\cdot y$ exists, whence so does $x\cdot D(y)$ by Lemma \ref{2p3} and so $(x,D(y))\in \C(P)$.  Let $u\in P$ be such that $y\cdot u$ exists; then $y\cdot D(u)$ exists by Lemma \ref{2p3}, and so $(y,e)\in \C(P)$ where $e=D(u)$, and $(x,D(y))\cdot (y,e)$ exists, with $(x,D(y))\rho=x$ and $(y,e)\rho=y$.  So $\rho$ is full.  Hence by Proposition \ref{semiquot}, $\C(P)/\ker(\rho)\cong P$. If $P$ is categorial, then $\rho$ is injective and hence is an isomorphism, since there is precisely one $e\in D(P)$ for which $s\cdot e$ exists.
\end{proof}

\begin{pro}  \label{canonfunc}
Let $\rho:P\rightarrow Q$ be a radiant between the
constellations $P$ and $Q$.  Then $F_{\rho}:\C(P)\rightarrow \C(Q)$ given by $(s,e)F_{\rho}=(s\rho,e\rho)$ is a functor.
\end{pro}  

\begin{proof}  We abbreviate $F_{\rho}$ to $F$ in what follows.
Suppose $(s,e)\circ (t,f)$ exists in $\C(P)$.  Then $D(t)=e$ and $s\cdot t$ exists in $P$, and indeed $s\cdot e$ and $t\cdot f$ exist.  So $s\rho\cdot t\rho$ exists in $Q$, as do $s\rho\cdot e\rho$ and $t\rho\cdot f\rho$, and $D(t\rho)=D(t)\rho=e\rho$.  Hence 
\begin{align*}
((s,e)F)\circ ((t,f)F)&=(s\rho,e\rho)\circ(t\rho,f\rho)\\
&=((s\rho)\cdot (t\rho),f\rho)\\
&=((s\cdot t)\rho,f\rho)\\
&=(s\cdot t,f)F\\
&=((s,e)\circ (t,f))F
\end{align*}
exists.  Moreover, 
\begin{align*}
D((s,e)F)&=D((s\rho,e\rho))\\
&=(D(s\rho),D(s\rho))\\
&=(D(s)\rho,D(s)\rho)\\
&=(D(s),D(s))F\\
&=(D((s,e))F,
\end{align*}
so $F$ is a radiant and hence a functor by Proposition \ref{radfunct}.
\end{proof}

We remark that the canonical extension construction on composable constellations has most of the properties of being a closure operator on a quasiordered set (``up to isomorphism" at least).  Thus defining $Q\lhd P$ if there is a full surjective radiant from composable constellation $P$ to composable constellation $Q$ (that is, $Q$ is isomorphic to a constellation quotient of $P$), we see that $\lhd$ is a quasiorder on the class of composable constellations.  Moreover, from the result above, the following are satisfied:
\bi
\item $P\lhd \C(P)$
\item $P\lhd Q$ implies $\C(P)\lhd \C(Q)$
\item $\C(\C(P))\cong \C(P)$.
\ei
So if one works with isomorphism classes of composable constellations, one obtains a closure operator on the resulting poset, and the closed elements are the (isomorphism classes of) categories.  It follows that for composable $P$, $\C(P)$ is the ``smallest" category ``above" $P$: so if there is a radiant $K\rightarrow P$ for some category $K$, then there must be a functor $F:K\rightarrow \C(P)$.  Indeed, the canonical extension satisfies a universal property, even for general constellations.

\begin{pro}  \label{unifunc}
If $P$ is a constellation, and $K$ is a category for which $f:K\rightarrow P$ is a radiant, then there is a unique functor $F:K\rightarrow \C(P)$ for which $F\rho=f$, where $\rho$ is as in Proposition \ref{canonfunc}.  If $f$ separates projections, then $F$ is strong.
\end{pro}
\begin{proof}
Define $aF=(af,R(a)f)$ for all $a\in K$.  Now for all $a,b\in K$ for which $a\cdot b$ exists in $K$, we have $R(a)=D(b)$, so $R(a)f=D(b)f$, and so 
$af\cdot R(a)f=af\cdot D(b)f=(a\cdot D(b))f=af$.  Hence
\begin{align*}
(a\cdot b)F&=((a\cdot b)f,R(a\cdot b)f)\\
&=(af\cdot bf,R(b)f)\\
&=(af,R(a)f)\circ (bf,R(b)f))\\
&=aF\circ bF,\mbox{ and}
\end{align*} 
\begin{align*}
D(a)F&=(D(a)f,R(D(a))f)\\
&=(D(a)f,D(a)f)\\
&=(D(af),D(af))\\
&=D((af,R(a)f))\\
&=D(aF),
\end{align*}
so $F$ is a radiant and hence by Proposition \ref{radfunct}, it is a functor.

Now for all $a\in K$, $aF\rho=(af,R(a)f)\rho=af$, so $F\rho=f$.  Conversely, suppose $G$ is a functor $K\rightarrow \C(P)$ for which $G\rho=f$.  Then for $a\in K$, suppose $aG=(b,e)\in \C(P)$.  Then $af=(b,e)\rho=b$, and $R(a)G=R(aG)=R((b,e))=(e,e)$.  So $R(a)f=R(a)G\rho=(e,e)\rho=e$, and so 
$$aG=(b,e)=(af,R(a)f)=aF,$$
showing that $F$ is unique.

Now assume $f$ separates projections.  Suppose that $aF\circ bF$ exists.  Then $(af,R(a)f)\circ (bf,R(b)f)$ exists, and so $af\cdot bf$ exists, and $R(a)f=D(bf)=D(b)f$, so since $f$ separates projections, $R(a)=D(b)$ and so $a\cdot b$ exists.  Hence $F$ is strong.
\end{proof}


In many categories, the notion of substructure can be captured in terms of arrows (this is true in any category of algebras for example), but not in all.  (We here use the term ``substructure" rather than ``subobject", since the latter has an existing definition in category theory in terms of equivalence classes of arrows that often but not always captures the intended meaning.)

Recall the constellation $P=\{e,f\}$ arising from the partial order in which $e<f$, so that $D(P)=P$ and $e\cdot f$ exists but $f\cdot e$ does not, so that $K=\C(P)=\{E,s,F\}$ where $E=(e,e),s=(e,f), F=(f,f)$ and $D(K)=\{E,F\}$.  We can concretely realise $P$ in terms of partial mappings on the set $X=\{x,y\}$ by letting $e=\{(x,x)\}$ and $f=\{(x,x),(y,y)\}$ with constellation composition and domain defined as usual in $\C_X$.  It is then natural to realise $\C(P)$ as cod-functions on $X$ in which $E$ is represented as the identity map on $\{x\}$, $F$ as the identity map on $\{x,y\}$, and $s$ as the map $\psi:\{x\}\rightarrow \{x,y\}$ in which $x\psi=x$ (a subcategory of SET).  This is a category in which $\{x\}$ may naturally be viewed as a substructure of $\{x,y\}$.

However, we can also realise $\C(P)$ in a different way as a (full) subcategory of SET, as follows: $E$ is represented as the identity on $\{x\}$, $F$ as the identity on $\{y\}$, and $s$ as $\psi:\{x\}\rightarrow \{y\}$ given by $x\psi=y$.  Of course now, $\{x\}$ is not a substructure of $\{y\}$ in any natural sense.
This shows that the concept of substructure is not category-theoretic.  However, it is constellation-theoretic, being captured by the assertion ``$e\cdot f$ exists" as in all of our earlier concrete examples (a formulation also far simpler than the usual category-theoretic definition of subobject, which as just shown may not yield the ``right" concept anyway).  In particular, if we have a representation of a category as $\C(P)$ for some constellation $P$, then we may use the standard quasiorder on $D(P)$ to induce one on $D(\C(P))$, and this provides a notion of substructure on the objects of $\C(P)$.

\subsection{Canonical congruences}

It is possible to give an internal description of the congruences on a category $K$ that give constellation quotients $P$ for which $K\cong \C(P)$.
Let us say that a congruence $\delta$ on a constellation $P$ is {\em canonical} if 
\bi
\item $\delta$ separates projections;
\item if $(a,b)\in \delta$ and $a\cdot e$ and $b\cdot e$ both exist for some $e\in D(P)$, then $a=b$.
\ei

\begin{pro}  \label{canonsimp}
An equivalence relation $\delta$ on a constellation $P$ is a canonical congruence if and only if it satisfies the following:
\bi
\item if $(a,b)\in \delta$ then $D(a)=D(b)$;
\item if $(a,b)\in \delta$ and $a\cdot e$ and $b\cdot e$ exist for some $e\in D(P)$, then $a=b$.
\item if  $(b,c)\in \delta$ and $a\cdot b$ (and hence $a\cdot c$) exists, then $(a\cdot b,a\cdot c)\in\delta$. 
\ei
\end{pro}
\begin{proof}
Obviously a canonical congruence satisfies all of the above.  Conversely, let $\delta$ be an equivalence relation on $P$ satisfying the above.  If $(e,f)\in\delta$ for some $e,f\in D(P)$, then $e=D(e)=D(f)=f$, so $\delta$ separates elements of $D(P)$.
Suppose $(a,b)\in\delta$ and $(c,d)\in \delta$.  Then $D(a)= D(b)$ and $D(c)=D(d)$.  Suppose $a\cdot c$ and $b\cdot d$ exist.  Then $a\cdot D(c)$ and $b\cdot D(d)=b\cdot D(c)$ exist, so $a=b$ by the second assumed condition.  So then we have $(a\cdot c,a\cdot d)\in\delta$ by the third.  Also, $D(a)=D(b)$, so $\delta$ trivially respects $D$.  Hence $\delta$ is a congruence, which is canonical by the second law above.
\end{proof}

We are most interested in canonical congruences on categories.

\begin{pro}  
If $K$ is a category, then the congruence $\delta$ on the constellation reduct of $K$ is canonical if and only if 
\bi
\item $\delta$ separates projections;
\item if $(a,b)\in \delta$ and $R(a)=R(b)$ then $a=b$.
\ei
\end{pro}
\begin{proof}
If the second condition above holds and $a\ \delta\ b$ with $a\cdot e$ and $b\cdot e$ existing, then $e=R(a)=R(b)$, so $a=b$.  The converse is immediate.
\end{proof}

By a similar argument, we have the following consequence of Proposition \ref{canonsimp}.
 
\begin{pro}  \label{canonsimpcat}
An equivalence relation $\delta$ on a category $K$ is a canonical congruence if and only if it satisfies the following:
\bi
\item if $(a,b)\in \delta$ then $D(a)=D(b)$;
\item if $(a,b)\in \delta$ and $R(a)=R(b)$, then $a=b$;
\item if $(b,c)\in \delta$ and $a\cdot b$ (and hence $a\cdot c$) exists, then $(a\cdot b,a\cdot c)\in\delta$. 
\ei
\end{pro}

\begin{pro}  \label{CPsim}
If $P$ is a constellation, then the relation $\sim$ on $\C(P)$ given by $(a,e)\sim (b,f)$ if and only if $a=b$ is a canonical congruence, and if $P$ is composable, then $P\cong C(P)/{\sim}$.
\end{pro}
\begin{proof}
Let $P$ be a constellation, and recall the radiant $\rho:\C(P)\rightarrow P$ defined by setting $(s,e)\rho=s$ for all $(s,e)\in \C(P)$.  Clearly $\ker(\rho)={\sim}$, so the latter is a congruence.  Suppose $(a,e)\sim(b,f)$.  Then of course $a=b$, so $D(a)=D(b)$ and so $D((a,e))=D((b,f))$.  If also $R((a,e))=R((b,f))$ then $(e,e)=(f,f)$, so $e=f$, and so since also $a=b$, we have $(a,e)=(b,f)$.  So $\sim$ is a canonical congruence on $C(P)$.

If $P$ is composable then $\rho$ is full and surjective by Proposition \ref{radim}, and $\C(P)/{\sim}=\C(P)/\ker(\rho)\cong P$ as constellations by Proposition \ref{semiquot}.  
\end{proof}

We have the following easy corollary of Corollary \ref{sepquot}.

\begin{pro}  \label{canonquot}
Let $P$ be a constellation, $\delta$ a canonical congruence on $P$.  Then $Q=P/\delta$ is a constellation for which $D(Q)=\{[e]\mid e\in D(P)\}$, which is composable if $P$ is.
\end{pro}

We call the quotient $Q$ in the above proposition a {\em canonical quotient} of $P$.

As noted earlier, if $K$ is a category then $K\cong \C(K)\cong \C(K/\triangle)$, and the diagonal relation $\triangle$ is a canonical congruence on (the constellation reduct of) $K$.  More generally, we have the following.

\begin{thm}  \label{canoniso}
Let $K$ be a category, $\delta$ a canonical congruence on (the constellation reduct of) $K$.  Letting $P$ be the canonical quotient $K/\delta$ (a composable constellation by Proposition \ref{canonquot}), $K\cong \C(P)$ via the isomorphism given by $s\mapsto ([s],[R(s)])$ where $[x]$ is the $\delta$-class containing $x\in K$.
\end{thm}
\begin{proof}
We show that the mapping $\psi:K\rightarrow \C(P)$ given by $s\psi=([s],[R(s)])$ is an isomorphism of categories.  (Note that $[R(s)]\in D(P)$ for any $s\in K$, so $\psi$ is well-defined.)  Now if $f:K\rightarrow K/\delta$ is the natural map, then it is full and separates projections by Proposition \ref{sepcorresp}, and $s\psi=(sf,R(s)f)$ for all $s\in K$, so by Proposition \ref{unifunc} and its proof, $\psi$ is a strong functor.

If $a\psi=b\psi$ then $([a],[R(a)])=([b],[R(b)])$, so $[a]=[b]$ and $R(a)=R(b)$, giving $a=b$ by one of the properties of canonical congruences.  Hence $\psi$ is one-to-one.  For $([a],[e])\in \C(P)$, we have that $[a]\circ [e]$ exists, so $a_1\cdot e$ exists for some $a_1\in [a]$ since every $x\in[e]$ has $D(x)=D(e)=e$.  So $R(a_1)=e$ and so $a_1\psi=([a_1],[R(a_1)])=([a],[e])$.  Hence $\psi$ is a bijection, and so is an isomorphism.
\end{proof}

Since any congruence on the constellation $P$ contained in a canonical one is itself canonical, it follows that the canonical congruences on $P$ are closed under arbitrary intersection and so form a complete semilattice, though not necessarily a lattice (as we show below).  

The full relation is a congruence on any constellation.  However, it is only rarely canonical.

\begin{pro}
The full relation $\triangledown$ on a constellation $P$ is a canonical congruence if and only if $|D(P)|=1$, and for all $s\in P$, the only composable element of $P$ is the unique element of $D(P)$.
\end{pro}
\begin{proof}
If $\triangledown$ is a canonical congruence on the constellation $P$, then for all $s,t\in P$, $D(s)=D(t)$, so $D(P)$ has only one element, $e$ say.  But also, for all $s,t\in P$, if $s\neq t$ then there is no $e\in D(P)$ for which $s\cdot e$ and $t\cdot e$ exists.  So in particular, if $s\cdot t$ exists, then $s\cdot D(t)=s\cdot e$ exists, so because $e\cdot e$ exists, it must be that $s=e$.  

Conversely, suppose $|D(P)|=1$ and $s\cdot t$ existing implies $s=D(t)$.  Then $D(s)=D(t)$ for all $s,t\in P$.  Also, suppose $s\cdot e$ and $t\cdot e$ exist.  Then $s=t=D(e)=e$.  So $\triangledown$, being a congruence, is a canonical congruence.
\end{proof}

\begin{cor} If $K$ is a category, the congruence $\triangledown$ is canonical if and only if $K$ has a single element.
\end{cor}
\begin{proof} If $K$ is a category on which $\triangledown$ is canonical, then for all $s\in K$, $s\ \triangledown\ R(s)$ and $R(s)=R(R(s))$, and so $s=R(s)\in D(K)$.  The converse is immediate.
\end{proof}

It is possible to describe when two elements of a category can be related by a canonical congruence.  

Let $K$ be a category.  For $a\in K$, define $\ker(a)=\{(x,y)\in K\times K\mid \mbox{$x\cdot a$ and $y\cdot a$ exist and are equal}\}$. For unequal $a,b\in K$, define the relation $\delta_{a,b}$ on $K$ by setting $(s,t)\in\delta_{a,b}$ if and only if $s=t$ or $\{s,t\}=\{x\cdot a,x\cdot b\}$ for some $x\in K$.

\begin{pro}  \label{pcc}
Let $K$ be a category, with $a,b\in K$, $a\neq b$.  Then there is a canonical congruence $\delta$ on $K$ for which $(a,b)\in \delta$ if and only if
\bi
\item $D(a)=D(b)$;
\item $R(a)\neq R(b)$; and
\item $\ker(a)=\ker(b)$.
\ei
In this case, $\delta_{a,b}$ is the smallest canonical congruence relating $a,b$.
\end{pro}
\begin{proof}
Suppose first that $\delta$ is a canonical congruence on $K$, with $(a,b)\in\delta$.  Then of course $D(a)=D(b)$ and $R(a)\neq R(b)$.  Suppose $x\cdot a$ and $y\cdot a$ exist and are equal.  Then $x\cdot b$ and $y\cdot b$ exist since $D(a)=D(b)$, and so 
$x\cdot b \ \delta\ x\cdot a=y\cdot a\ \delta\ y\cdot b$ since $\delta$ is a congruence.  Each has range $R(b)$ and so $x\cdot b=y\cdot b$ since $\delta$ is canonical.  So $\ker(a)\subseteq \ker(b)$. Symmetry now implies the opposite inclusion.

Conversely, suppose $a,b$ satisfy the three conditions.  Then the relation $\delta=\delta_{a,b}$ is obviously reflexive and symmetric.  Suppose $(s,t)\in \delta$ and $(t,u)\in \delta$.  If $s=t$ or $t=u$, then obviously $(s,u)\in \delta$, so suppose neither equality holds.  Then there are $x,y\in K$ for which (without loss of generality) $s=x\cdot a$, $t=x\cdot b$ and either (i) $t=y\cdot a$ and $u=y\cdot b$ or (ii) $t=y\cdot b$ and $u=y\cdot a$.  If (i), then  $t=x\cdot b=y\cdot a$, so $R(t)=R(b)=R(a)$, a contradiction.  So (ii) holds and $t=x\cdot b=y\cdot b$, giving $s=x\cdot a=y\cdot a=u$ and so trivially $(s,u)\in\delta$.  So $\delta$ is transitive and hence is an equivalence relation.  

Clearly, $(x,y)\in \delta$ implies $D(x)=D(y)$, and clearly $R(x)=R(y)$ implies $x=y$.  Suppose $s\in K$ and $(u,v)\in\delta$, with $s\cdot u$ and hence $s\cdot v$ existing.  If $u=v$ then trivially $s\cdot u\ \delta\ s\cdot v$.  If not, then without loss of generality, $u=y\cdot a$ and $v=y\cdot b$ for some $y\in K$, and then
$s\cdot (y\cdot a)$ exists and equals $(s\cdot y)\cdot a$, while $s\cdot (y\cdot b)$ exists and equals $(s\cdot y)\cdot b$, which are related by $\delta$ by definition.  So by Proposition \ref{canonsimpcat}, $\delta$ is a canonical congruence, and $(a,b)\in \delta$ since $a=D(a)\cdot a$ and $b=D(a)\cdot b$.

If $\tau$ is another canonical congruence relating $a,b$ and $(s,t)\in \delta$, then either $s=t$ or $s=x\cdot a$ and $t=x\cdot b$ (or the other way around), and so $s\ \tau\ t$ also.  So $\delta\subseteq \tau$.
\end{proof}

Let $K$ be a category.  For unequal $a,b\in K$ satisfying the conditions of Proposition \ref{pcc}, it is appropriate to describe $\delta_{a,b}$ as the {\em principal canonical congruence} on $K$ generated by $a,b$.



Let us call a constellation with no canonical congruences except the diagonal relation {\em canonically simple}.  Every quasiordered set viewed as a constellation is canonically simple since the only projection-separating congruence on it is the diagonal relation. 

\begin{pro}  \label{ranub}
If the constellation $P$ is such that for every $s,t\in P$ for which $D(s)=D(t)$, there exists $e\in D(P)$ for which $s\cdot e$ and $t\cdot e$ exist, then $P$ is canonically simple.
\end{pro}
\begin{proof}
Let $P$ be a constellation satisfying the above conditions, with $\delta$ a canonical congruence on $P$.  Suppose $s\ \delta\ t$.  Then $D(s)=D(t)$ by Proposition \ref{canonsimp}, so by assumption, there exists $e\in D(P)$ for which $s\cdot e$ and $t\cdot e$ exist.  So again by Proposition \ref{canonsimp}, $s=t$.  So $\delta=\triangle$ and so $P$ is canonically simple.
\end{proof}

In particular, if a constellation has a global right identity element then it is canonically simple.  So $\C_X$ is canonically simple for any non-empty set $X$. 

It follows from Proposition \ref{ranub} that if the category $K$ is such that for every $a,b\in K$, $D(a)=D(b)$ implies that $R(a)=R(b)$, then $K$ is canonically simple; for example, every monoid is canonically simple.  More generally, we have the following consequence of Proposition \ref{pcc}.

\begin{cor}
The category $K$ is canonically simple if and only if for all $a,b\in K$, if $D(a)=D(b)$ and $\ker(a)=\ker(b)$ then $R(a)=R(b)$.
\end{cor}
  
For two surjective maps $f:A\rightarrow B$ and $g:A\rightarrow C$, there is an identity mapping $1_D:D\rightarrow D$ for which $f\cdot 1_D$ and $g\cdot 1_D$ exist in $CSET$ (for example, let $D=B\cup C$), so by Proposition \ref{ranub}, $CSET$ is canonically simple.  By contrast, we have the following.

\begin{pro} 
$CGRP$ is not canonically simple.  
\end{pro}
\begin{proof}
Pick two projections $e,f\in D(CGRP)$ which are the identity map on two distinct trivial groups $G,H$, and define $\delta$ on $CGRP$ as follows: $s\ \delta\ t$ if and only if $s=t$ or $D(s)=D(t)$ and one of $s,t$ maps onto $G$ and the other onto $H$.  An easy case analysis shows $\delta$ is an equivalence relation.  (For example, if $(s,t)\in\delta$ and $(t,u)\in\delta$ with $s\neq t$ and $t\neq u$, then $D(s)=D(t)=D(u)$.  If without loss of generality $s$ maps onto $H$, then so must $u$ and so $s=u$.)  Clearly $(s,t)\in\delta$ implies $D(s)=D(t)$.  If $s\ \delta\ t$ and $s\neq t$ then there is no $g\in D(CSET)$ for which $s\cdot g$ and $t\cdot g$ exist since then the group on which $g$ is the identity map would have two distinct idempotent elements.  Finally, if $(s,t)\in\delta$ and $u\in CSET$ is such that $u\cdot s$ and $u\cdot t$ exist, then obviously $(u\cdot s,u\cdot t)\in \delta$.  So by Proposition \ref{canonsimp}, $\delta$ is a canonical congruence.
\end{proof}

Similar arguments apply to other concrete constellations arising in algebra.

A {\em maximal} canonical congruence on a constellation is a canonical congruence that does not lie inside another canonical congruence.  

\begin{pro}  \label{maxcanon}
Every constellation has a maximal canonical congruence. 
\end{pro}
\begin{proof}
The diagonal relation on the constellation $P$ is a canonical congruence.  The union of any chain of canonical congruences is easily seen to itself be a canonical congruence, so the result follows by Zorn's Lemma (for partially ordered classes rather than sets, a strictly stronger axiom of mathematics). 
\end{proof}

Proposition \ref{trans} specialises to canonical congruences.  Again, for $\delta$ an equivalence relation on the constellation $P$, denote by $[s]_{\delta}$ the $\delta$-class containing $s\in P$.

\begin{pro}  \label{transcanon}
Let $P$ be a constellation with $\delta$ a canonical congruence on it, and let $Q=P/\delta$.  Then for any canonical congruence $\epsilon$ on $Q$, the congruence $\gamma$ on $P$ that contains $\delta$ and such that $P/\gamma\cong Q/\epsilon$ as in Proposition \ref{trans} is canonical.  Conversely, if $\gamma$ is a canonical congruence on $P$ containing $\delta$, then the congruence $\epsilon$ on $Q$ for which $P/\gamma\cong Q/\epsilon$ as in Proposition \ref{trans} is canonical.  
\end{pro}
\begin{proof}
Recall that $Q$ is a constellation in which $D(Q)=[e]_{\delta}$ for some $e\in D(P)$, by Corollary \ref{sepquot}.

Let $\epsilon$ be a canonical congruence on $Q$.  Define $\gamma$ on $P$ by setting $(s,t)\in\gamma$ if and only if $[s]_{\delta}\ \epsilon\ [t]_{\delta}$, a projection-separating congruence by Proposition \ref{trans}.  Suppose $a\ \gamma\ b$ and there is $e\in D(P)$ for which $a\cdot e$ and $b\cdot e$ both exist.  Thus, $[a]_{\delta}\ \epsilon\ [b]_{\delta}$, and both $[a]_{\delta}\cdot [e]_{\delta}$ and $[b]_{\delta}\cdot [e]_{\delta}$ exist, with $[e]_{\delta}\in D(Q)$. So $[a]_{\delta}=[b]_{\delta}$ since $\epsilon$ is canonical, and so $a=b$ since $\delta$ is.  Hence $\gamma$ is canonical.

Conversely, suppose $\gamma$ is a canonical congruence on $P$, with $\delta\subseteq \gamma$.  Consider the projection-separating congruence $\epsilon$ on $Q$ defined by setting $[s]_{\delta}\ \epsilon\ [t]_{\delta}$ if and only if $s\ \gamma\ t$ (as in Proposition \ref{trans}).  Suppose $[s]_{\delta}\ \epsilon\ [t]_{\delta}$ and $[s]_{\delta}\cdot [e]_{\delta}$ and $[t]_{\delta}\cdot [e]_{\delta}$ exist in $Q$ for some $e\in D(P)$. that is, $[e]_{\delta}\in D(Q)$.  Then $s\ \gamma\ t$ and $s_1\cdot e$ and $t_1\cdot e$ exist for some $s_1\in [s]_{\delta}$ and $t_1\in [t]_{\delta}$ since $\delta$ separates projections and hence is right strong by Proposition \ref{seproj}, so we have $s_1\ \delta\ s\ \gamma\ t\ \delta\ t_1$, and so since $\delta\subseteq \gamma$, we have $s_1\ \gamma\ t_1$.  Since $\gamma$ is canonical, $s_1=t_1$ and so $s\ \delta\ s_1=t_1\ \delta\ t$, so $[s]_{\delta}=[t]_{\delta}$.  So $\epsilon$ is canonical.
\end{proof}

\begin{cor}  \label{maxsimp}
Let $P$ be a constellation with $\delta$ a canonical congruence on it.  Then $P/\delta$ is canonically simple if and only if $\delta$ is maximal.
\end{cor}

\begin{cor}
Every constellation has a canonically simple canonical quotient.
\end{cor}

\begin{cor}
If $P$ is a composable constellation and $\delta$ a canonical congruence on $P$, then $\C(P)\cong \C(P/\delta)$.
\end{cor}
\begin{proof}
If $\delta$ is a canonical congruence on $P$ then since $P$ is a canonical quotient of $\C(P)$ by Proposition \ref{CPsim}, $P/\delta$ is a canonical quotient of $\C(P)$ by Proposition \ref{transcanon}, and so $\C(P)\cong C(P/\delta)$ by Theorem \ref{canoniso}.
\end{proof}

In particular, given a category $K$ of interest, we seek a canonically simple canonical quotient $P$ of $K$, since such $P$ is as ``simple" as possible with the property that $K\cong \C(P)$.

\subsection{Some examples}

\begin{eg} The concrete constellations of Section \ref{examples} revisited. \end{eg}

Recall Example \ref{eg:CSET}: in the category $SET$, if $f:A\rightarrow B$ and $g:A\rightarrow C$, define $(f,g)\in \delta$ if and only if Im$(f)=$ Im$(g)$, and $xf=xg$ for all $x\in A$.  Clearly $\delta$ is an equivalence relation, and the three conditions in Proposition \ref{canonsimp} are easily seen to be satisfied, so $\delta$ is a canonical congruence on $K$ by that result.  In fact each $\delta$-class has a unique representative having smallest possible range: $f:A\rightarrow B$ is $\delta$-related to $\bar{f}:A\rightarrow $ Im$(f)$, defined to agree with $f$ on their common domain but to have codomain equal to the image of $f$.  This is a typical surjective function in the category, and indeed a typical element of $CSET$, as we have seen.  Now it is easy to see that for $f,g\in SET$, $[f]\cdot[g]$ exists in $SET/\delta$ if and only if 
Im$(\bar{f})\subseteq$ Dom$(g)=$ Dom$(\bar{g})$, if and only if $\bar{f}\cdot \bar{g}$ exists in $CSET$, and then equals $[h]$ where $\bar{h}=\bar{f}\cdot\bar{g}$ as computed in $CSET$.  It follows that $SET/\delta\cong CSET$ via the isomorphism $[s]\mapsto \bar{s}$, and so $SET\cong \C(CSET)$ as categories.  (As noted earlier, $CSET$ is canonically simple.)

A simplified version of this argument shows that for any set $X$, defining $\delta$ as for $SET$, $COD_X/\delta\cong \C_X$, and so $COD_X\cong \C(\C_X)$ as categories.  (As noted earlier, $\C_X$ is canonically simple since it has a global right identity element.)  It is of interest that for finite $X$ having $n$ elements, $COD_X$ generally has $\sum_{m=0}^n {n\choose m}(m+1)^n$ elements, as can be seen by adding up the number of partial functions having domain of size $n$ and range varying from empty to size $n$, while $\C_X$ has only $(n+1)^n$.   For $n=1,2,3$, these two numbers are $18$ and $9$, $170$ and $64$, and $2200$ and $625$ respectively, and the ratio of the latter to the former tends to zero as $n$ increases.  

A similar argument can be given for the category of groups: one can define $\delta$ analogously, noting that Im$(f)$ is itself a group and so $\bar{f}$ as defined above is a group homomorphism, the rest of the argument proceeding as above for $SET$, leading to the conclusion that $GRP\cong \C(CGRP)$ as categories.  The same line of argument applies to the categories of rings, modules over a ring, semigroups and so on.

In fact the same relationship exists between $TOP$ and $CTOP$.  This is because every  continuous function $f:X\rightarrow Y$ is determined by the continuous function $\bar{f}:X\rightarrow$ Im$(f)$ (where Im$(f)$ is the subset of $Y$ consisting of the image of $f$ equipped with the subspace topology), as above: indeed $f$ is the composite of $\bar{f}$ with the obvious embedding of Im$(f)$ into $Y$.  As a result, we may define a canonical congruence $\delta$ as for $SET$ and $GRP$, with each $\delta$-class in $TOP$ containing a unique element of $CTOP$, and indeed $CTOP\cong TOP/\delta$, so that $TOP\cong \C(CTOP)$.  Similarly for many other familiar concrete categories.

For the category $SET^{\infty}$, we may define $\delta$ as for $SET$.  Not every $\delta$-class has its canonical representative within $SET^{\infty}$, but the restriction of $\delta$ to $SET^{\infty}$ is easily seen to be a canonical congruence on it (in general the restriction of a canonical congruence on a category $K$ to a subcategory $L$ is a canonical congruence on $L$).  Then each $[s]\in SET^{\infty}/\delta$ may have associated with it the surjective map $\bar{s}\in CSET$ defined as for $SET$.  So since the two subconstellations $SET^{\infty}/\delta$ of $SET/\delta$ and $CSET^{\infty}$ of $CSET$ correspond under the earlier isomorphism $[s]\mapsto \bar{s}$ between $SET/\delta$ and $CSET$, $SET^{\infty}/\delta\cong CSET^{\infty}$, and so $SET^{\infty}\cong \C(CSET^{\infty})$.  Similarly for $GRP^{\infty}$ and $CGRP^{\infty}$, and so on.  

\begin{eg} A six-element category. \end{eg}

Consider the category with six elements $K=\{e,f,g,a,b,c\}$, in which $D(K)=\{e,f,g\}$, and $D(a)=D(b)=D(c)=e$, $R(a)=R(b)=f$ and $R(c)=g$.  The partial multiplication table for $K$ is as follows:

$$
\begin{array}{c|cccccc}
\cdot&b&c&e&a&f&g\\
\hline
b&-&-&-&-&b&-\\
c&-&-&-&-&-&c\\
e&b&c&e&a&-&-\\
a&-&-&-&-&a&-\\
f&-&-&-&-&f&-\\
g&-&-&-&-&-&g
\end{array}
$$

It is easy to check from Corollary \ref{canonsimpcat} and the above that the equivalence relations $\delta_1$ and $\delta_2$ of $K$ giving the following partitions are canonical congruences:
$$\{b,c,e\},\ \{a\},\ \{f\},\ \{g\}\mbox{ and }\{b,c\},\ \{e,a\},\ \{f\},\ \{g\}.$$
Indeed both are maximal since $e,f,g$ must be in distinct classes (since canonical congruences separate projections), and $a,b$ cannot be related by a canonical congruence since their ranges are equal, so not all three can be put in a class with $e$ (the only projection they could be grouped with since domains of congruent elements must agree).  So the poset of canonical congruences of a constellation (or even a category) need not be a lattice. 

The constellations $K/\delta_1$ and $K/\delta_2$ are canonically simple by Corollary \ref{maxsimp}, and  have partial multiplication tables as follows (writing singleton sets as the element they contain):

$$
\begin{array}{c|cccc}
\cdot&[e]&a&f&g\\
\hline
[e]&[e]&a&[e]&[e]\\
a&-&-&a&-\\
f&-&-&f&-\\
g&-&-&-&g
\end{array}\
\mbox{ and }\
\begin{array}{c|cccc}
\cdot&[b]&[e]&f&g\\
\hline
[b]&-&-&[b]&[b]\\[0pt]
[e]&[b]&[e]&[e]&-\\
f&-&-&f&-\\
g&-&-&-&g
\end{array}.
$$
\medskip

Then $D(K/\delta_1)=\{[e],f,g\}$ with $[e]\lesssim f,g$, while $D(K/\delta_2)=\{[e],f,g\}$ with $[e]\lesssim f$ only, so these constellations are not isomorphic.  
Note that $\C(K/\delta_1)=\{([e],f), ([e],g),([e],[e]), (a,f),(f,f),(g,g)\}$, and under the isomorphism of Theorem \ref{canoniso}, $b\mapsto ([e],f), c\mapsto ([e],g), e\mapsto  ([e],[e]), a\mapsto (a,f), f\mapsto (f,f)$, and $g\mapsto (g,g)$, 

\begin{eg} The category of surjective functions on a two-element set. \label{surjeg} 
\end{eg}

The maximal canonical congruences in the previous example both yield normal constellation quotients.  Now let $X$ be a set and consider the subcategory ${\mc S}_X$ of $COD_X$ consisting of the {\em surjective} non-empty cod-functions on $X$ (so the elements are the same as those in $\C_X$ except that a codomain determined by its range is defined for every element, the empty function is absent, and multiplication is more restricted).  It is easy to see that $\delta$ as defined earlier on ${\mc P}_X$ is the diagonal relation when restricted to ${\mc S}_X$.  However, ${\mc S}_X$ is not canonically simple.  For example, let $X=\{a,b\}$, so that ${\mc S}_X$ has eight elements (the eight non-empty partial maps on $X$), namely
$$\{1,f_a,f_b,i,1_a,ab,1_b,ba\},$$
where we have $1=\{(a,a),(b,b)\}$, $f_a=\{(a,a),(b,a)\}$, $f_b=\{(a,b),(b,b)\}$, $i=\{(a,b),(b,a)\}$, $ab=\{(a,b)\}$, $ba=\{(b,a)\}$, $1_a=\{(a,a)\}$, and $1_b=\{(b,b)\}$.  $D$ is defined in the obvious way and $R$ in terms of images, and the partial multiplication table is as follows:

$$
\begin{array}{c|ccccccccc}
\cdot&1&f_a&f_b&i&1_a&ab&1_b&ba\\
\hline
1&1&f_a&f_b&i&-&-&-&-\\
f_a&-&-&-&-&f_a&f_b&-&-\\
f_b&-&-&-&-&-&-&f_b&f_a\\
i&i&f_a&f_b&1&-&-&-&-\\
1_a&-&-&-&-&1_a&ab&-&-\\
ab&-&-&-&-&-&-&ab&1_a\\
1_b&-&-&-&-&-&-&1_b&ba\\
ba&-&-&-&-&ba&1_b&-&-\\
\end{array}
$$

By Proposition \ref{maxprojsep}, the largest projection-separating congruence on ${\mc S}_X$ for any finite $X$ has, for all $e,f\in D({\mc S}_X)$, $[e]\lesssim [f]$ if and only if there exist $a,b\in {\mc S}_X$ for which $D(a)=e$, $D(b)=f$ and $a\cdot b$ exists in ${\mc S}_X$, that is, $a\cdot f$ exists in ${\mc S}_X$.  This says that the subset of $X$ corresponding to $e$ has at least as many elements as that corresponding to $f$.  So in the current case, $[1]$ is below both $[1_a]$ and $[1_b]$.

The equivalence relation $\gamma$ that gives rise to the following partition of ${\mc S}_X$ is easily checked to be a canonical congruence using Proposition \ref{canonsimpcat}, and indeed is the largest canonical congruence on it:
$$\{1\}, f=\{f_a,f_b\}, \{i\}, [1_a]=\{1_a,ab\}, [1_b]=\{1_b,ba\}.$$
So again identifying singletons with their unique elements, the constellation $P={\mc S}_X/\gamma=\{1,f,i,[1_a],[1_b]\}$ is canonically simple, and has multiplication table as follows.

$$
\begin{array}{c|ccccc}
\cdot&1&[1_a]&[1_b]&f&i\\
\hline
1&1&-&-&f&i\\[0pt]
[1_a]&-&[1_a]&[1_a]&-&-\\[0pt]
[1_b]&-&[1_b]&[1_b]&-&-\\
f&-&f&f&-&-\\
i&i&-&-&f&1
\end{array}
$$
From this it can be seen that $D(P)=\{1,[1_a],[1_b]\}$ is the set of right identities in $P$ (corresponding as it must to $\{1,1_a,1_b\}$ in ${\mc S}_X$), and we have that $[1_a]\approx [1_b]$ under the standard quasiorder on $D(P)$, so $P$ is not normal. 

The equivalence relation in which $\{f_a,f_b\}$ is a class and all other partition classes are singletons is the unique non-trivial canonical congruence on ${\mc S}_X$ having normal quotient, and for this quotient, the standard quasiorder on the projections is flat.  (In general, one can prove using Zorn's Lemma that a maximal canonical congruence with normal quotient always exists.)  So in particular, there is no constellation quotient $P$ of ${\mc S}_X$ from which the obvious substructure partial order on the objects of ${\mc S}_X$ derives (that is, $1_a\lesssim 1$, $1_b\lesssim 1$), reflecting the fact that there is no suitable monic in ${\mc S}_X$ having domain $1_a$ and codomain $1$.  (Of course this partial order on the projections {\em is} present in the structure of $\C_X$.)

\subsection{$\delta$-categories and composable constellations}

As we have seen, information about substructure relationships can be lost in passing from $P$ to $\C(P)$, reflecting the fact that $P$ is not determined by the category $\C(P)$.  Both these issues can be dealt with by specifying the canonical congruence $\sim$ for which $\C(P)/{\sim}\cong P$.

Let us say that the pair $(K,\delta)$, consisting of a category $K$ equipped with a particular canonical congruence $\delta$ on it, is a {\em $\delta$-category}.  For a $\delta$-category $(K,\delta)$, we often just write $K$ if $\delta$ is understood.  In particular, for $P$ a constellation, the category $\C(P)$ is to be viewed as a $\delta$-category in which the canonical congruence is $\sim$ as in Proposition \ref{CPsim}, unless stated otherwise.

A {\em $\delta$-functor} $\psi$ from the $\delta$-category $(K_1,\delta_1)$ to the $\delta$-category $(K_2,\delta_2)$ is a functor from $K_1$ to $K_2$ which additionally respects the distinguished canonical congruences: if $(f,g)\in \delta_1$, then $(f\psi,g\psi)\in \delta_2$.  In this way, the class of $\delta$-categories is itself a category in which the arrows are $\delta$-functors.  An isomorphism of $\delta$-categories is therefore a $\delta$-functor with a $\delta$-functor inverse, so $\psi:(K_1,\delta_1)\rightarrow (K_2,\delta_2)$ is an isomorphism of $\delta$-categories if and only if is a category isomorphism $K_1\rightarrow K_2$ for which $(s,t)\in \delta_1$ if and only if $(s\psi,t\psi)\in \delta_2$.  

We have the following easy consequence of Theorem \ref{canoniso}.

\begin{cor}  \label{corcanon}
Let $(K,\delta)$ be a $\delta$-category.  Then $(K,\delta)\cong (\C(K/\delta),\sim)$ as $\delta$-categories.
\end{cor}
\begin{proof} By Theorem \ref{canoniso}, $\psi:K\rightarrow \C(K/\delta)$ given by $s\psi=([s],[R(s)])$ is a constellation (hence category) isomorphism.  Also, for all $a,b\in K$, the following are equivalent: $a\ \delta\ b$; $[a]=[b]$; $([a],[R(a)])\sim ([b],[R(b)])$; $a\psi\sim b\psi$.  So $\psi$ is a $\delta$-category isomorphism.
\end{proof}

From the above and Proposition \ref{CPsim}, we obtain the following.

\begin{cor}  \label{main}
The maps under which $P\mapsto\C(P)$ and $K\mapsto K/\delta$, where $P$ is a composable constellation and $K$ is a $\delta$-category, are mutually inverse up to isomorphism: $\C(P)/{\sim} \cong P$, and $(\C(K/\delta),\sim)\cong (K,\delta)$ as $\delta$-categories.
\end{cor}

\begin{thm}  \label{main2}
The categories of composable constellations and $\delta$-categories are equivalent. 
\end{thm}
\begin{proof}
Let $\phi$ be the mapping that takes the radiant $\rho:P\rightarrow Q$ in the category of constellations to the functor $F_{\rho}$ in the category of categories, defined in Proposition \ref{canonfunc} (where it is shown to be a functor); in particular, the object $P$ (represented by the identity radiant on $P$) is mapped to its canonical extension category $\C(P)$ (again, represented by the identity functor).  In fact $F_{\rho}$ is a $\delta$-functor, since if $(s,e)\ \sim\ (t,f)$ in $\C(P)$, then $s=t$ so $s\rho=t\rho$, and so $(s\rho,e\rho)\ \sim\ (t\rho,f\rho)$ in $\C(Q)$.

Conversely, let $\psi$ be the mapping that takes the $\delta$-functor $f:K\rightarrow L$ in the category of $\delta$-categories to the mapping $\rho_f:K/\delta\rightarrow L/\delta$ given by $[s]\rho_f=[sf]$ (where $[s]$ denotes the $\delta$-class containing $s$).  This is well-defined since $f$ respects $\delta$.  Suppose $[s]\cdot [t]$ exists in $K/\delta$.  Then $s_1\cdot t$ exists for some $s_1\in [s]$ since canonical congruences are projection-separating and hence right strong by Proposition \ref{seproj}.
So $s_1f\cdot tf$ exists, and since $(sf,s_1f)\in \delta$, we have that $[sf]\cdot [tf]$ exists, and so $([s]\rho_f)\cdot ([t]\rho_f)$ exists.  Now
$$([s]\cdot[t])\rho_f=([s_1]\cdot[t])\rho_f=[(s_1\cdot t)f]=[s_1f\cdot tf]=[sf]\cdot [tf]=([s]\rho_f)\cdot([t]\rho_f)$$ 
since $(sf,s_1f)\in \delta$.  Moreover, $(D([s]))\rho_f=[D(s)f]=[D(sf)]=D([sf])=D([s]\rho_f)$.  So $\rho_f$ is a radiant.

We next show that each of $\phi$ and $\psi$ is a functor, beginning with $\phi$.  Suppose $\rho:P\rightarrow Q$ and $\tau:Q\rightarrow R$ are radiants between constellations $P,Q,R$.  Then $F_{\rho}:\C(P)\rightarrow \C(Q)$ and  $F_{\tau}: \C(Q)\rightarrow \C(R)$ are functors with composite $F_{\rho}F_{\tau}:\C(P)\rightarrow \C(R)$, and 
\begin{align*}
(s,e)(F_{\rho}F_{\tau})&=(s,e)F_{\rho}F_{\tau}\\
&=(s\rho,e\rho)F_{\tau}\\
&=(s\rho\tau,e\rho\tau)\\
&=(s(\rho\tau),e(\rho\tau))\\
&=(s,e)F_{\rho\tau}.
\end{align*}
Also, $F_{D(\rho)}:\C(P)\rightarrow \C(P)$ (since $D(\rho):P\rightarrow P$), and 
$$(s,e)F_{D(\rho)}=(sD(\rho),eD(\rho))=(s,e),$$
so $F_{D(\rho)}$ is the identity map on $\C(P)$, the domain of $F_{\rho}$, which is precisely what $D(F_{\rho}):\C(P)\rightarrow \C(P)$ is.   So $\phi$ is a radiant between two categories, hence a functor by Proposition \ref{radfunct}.  

Now suppose $(K,\delta_1),(L,\delta_2),(M,\delta_3)$ are $\delta$-categories, with $f:K\rightarrow L$ and $g:L\rightarrow M$ both $\delta$-functors.  Then $\rho_{f}:K/\delta_1\rightarrow L/\delta_2$ and $\rho_{g}:L/\delta_2\rightarrow M/\delta_3$ are radiants with composition $\rho_{f}\rho_{g}:K/\delta_1\rightarrow M/\delta_3$.  Then for each $[s]\in K/\delta_1$, 
$$[s](\rho_{f}\rho_{g})=[s]\rho_{f}\rho_{g}=[sf]\rho_{g}=[sfg]
=[s(fg)]=[s]\rho_{fg},$$
and $\rho_{D(f)}:K/\delta_1\rightarrow K/\delta_1$ (since $f:K/\delta_1\rightarrow K/\delta_2$) is such that for all $[s]\in K/\delta_1$, $[s]\rho_{D(f)}=[sD(f)]=[s]$, so 
$\rho_{D(f)}$ is nothing but the identity map on the domain of $\rho_{f}$, which is precisely what $D(\rho_{f})$ is.  So $\psi$ is a radiant between categories, hence a functor.

It remains to check that $\phi,\psi$ are mutually inverse up to isomorphism (when acting on objects).  But for a composable constellation $P$, $P\phi\psi=\C(P)\psi\cong P$, while for any $\delta$-category $(K,\delta)$, we have $(K,\delta)\psi\phi=\C(K/\delta)\phi=(\C(K/\delta),\sim)\cong (K,\delta)$.
\end{proof}

Recall that $SET/\delta\cong CSET$ via the isomorphism $[s]\mapsto \bar{s}$, and so $SET\cong \C(CSET)$ as categories, and indeed by the above, $(SET,\delta)\cong (\C(CSET),\delta)$ as $\delta$-categories.  Likewise, $COD_X/\delta\cong \C_X$, and so $COD_X\cong \C(\C_X)$ as $\delta$-categories.  Similarly for the other cases: $GRP\cong \C(CGRP)$, $SET^{\infty}\cong \C(CSET^{\infty})$, $GRP^{\infty}\cong \C(CGRP^{\infty})$, and so on, not just as categories but as $\delta$-categories.

Because of the correspondence between composable constellations and $\delta$-categories via canonical extensions, the substructure notion present in a constellation is also present in a $\delta$-category: for $e,f\in D(K)$, $e$ is a substructure of $f$ if and only if there exists $s\in K$ for which $(e,s)\in\delta$ and $R(s)=f$.  (This is because of Theorem \ref{canoniso}: in the case of $\C(P)$ where $P$ is a constellation, $s=(e,f)$ (where $e,f\in D(P)$ are such that $e\cdot f$ exists) is the unique element $x$ of $\C(P)$ such that $x\sim (e,e)$ and $R(x)=(f,f)$, and exists if and only if $(e,e)\leq (f,f)$.)  Example \ref{surjeg} shows that some categories admit no $\delta$-category structure capable of capturing the natural substructure partial order on their objects.

However, a notion of substructure is not always sufficient to specify the distinguished canonical congruence on a $\delta$-category.  For example, consider the constellation $P=\{s,e,f,g\}$, in which $D(P)=\{e,f,g\}$ with $D(s)=g$ and $s\cdot e,s\cdot f$ both existing but no other products existing aside from those that must.  (For example, let $X=\{1,2,3\}$ with $P$ the subconstellation of $\C_X$ in which $s=\{(1,2),(3,2)\}$, $e=\{(1,1),(2,2)\}$, $f=\{(2,2),(3,3)\}$ and $g=\{(1,1),(3,3)\}$.) Then $K=\C(P)=\{(s,e),(s,f),(e,e),(f,f),(g,g)\}$ is a $\delta$-category with respect to $\sim$, in which $D(K)=\{(e,e),(f,f),(g,g)\}$, and there are four $\sim$-classes, namely $\{(s,e),(s,f)\}$, $\{(e,e)\}$, $\{(f,f)\}$ and $\{(g,g)\}$.  However, it supports another canonical congruence $\triangle$, in which the $\delta$-classes are just the singletons.  Precisely the same notion of substructure is determined by this latter choice of canonical congruence.

\section{Open questions}

For a number of familiar concrete categories $K$, we have described natural associated constellations $P$, often consisting of the surjective morphisms in $K$ but equipped with the more liberal constellation product, and $P$ has turned out to be a canonical quotient of $K$, since $K\cong \C(P)$.  In some cases such as $CSET$, $P$ is canonically simple, and so the representation of $K$ as $\C(P)$ is ``best possible".  Others such as $CGRP$ are not canonically simple.  There is interest in determining the canonical simplicity status of other such constellations, for example $CTOP$, as well is in explicitly finding a maximal canonical congruence $\delta$ on $P$ (and hence on $K\cong \C(P)$) when $P$ is not canonically simple, so as to obtain a ``best possible" representation of $K$ as $\C(P)$.

The main observation of this paper is that category theory as it applies to the familiar concrete categories of modern mathematics (which come equipped with natural notions of substructures and indeed are $\delta$-categories) may be subsumed by constellation theory.  There is therefore interest in examining how many of the big applications of category theory to mathematical topics such as topology actually appear when re-cast in the more parsimonious form of constellations.  Notions such as natural transformations, left and right adjoint functors, equalizers and so forth all need to be formulated in the setting of constellations.  Indeed even Theorem \ref{main2} itself may have a stronger formulation in terms of an equivalence of suitably defined constellations.


\begin{thebibliography}{99}
\bibitem{constell} V. Gould and C. Hollings, `Restriction semigroups and inductive constellations', {\em Comm. Algebra} 38, 261--287 (2009).
\bibitem{gratzer} G. Gr\"{a}tzer, {\em Universal Algebra}, 2nd Edition, Springer, New York (1979).
\bibitem{agree} M. Jackson and T. Stokes, `Agreeable semigroups', 
{\em J. Algebra} 266, 393--417 (2003). 
\bibitem{lawson1} M.V. Lawson, `Semigroups and ordered categories I: the reduced case', {\em J. Algebra} 141, 422--462 (1991).
\end{thebibliography}
\end{document}